%% file: main.tex
\documentclass[11pt,reqno]{amsart}

\usepackage{amsmath}
\usepackage{amsfonts}
\usepackage{amssymb}
\usepackage{amsthm}
\usepackage{mathrsfs}
\usepackage{bm}

\usepackage{enumerate}
\usepackage{geometry}
\usepackage{hyperref}

\newcommand{\RR}{\mathbf{R}}

\newcommand{\ZZ}{\mathbf{Z}}
\newcommand{\NN}{\mathbf{N}}

\newcommand{\PP}{\mathbf{P}}

\renewcommand{\SS}{\mathbf{S}}

\newcommand{\cA}{\mathcal{A}}

\newcommand{\cC}{\mathcal{C}}

\newcommand{\cH}{\mathcal{H}}

\newcommand{\cL}{\mathcal{L}}
\newcommand{\cM}{\mathcal{M}}

\DeclareMathOperator{\ind}{ind}
\DeclareMathOperator{\nul}{nul}

\DeclareMathOperator{\support}{spt}

\DeclareMathOperator{\dist}{dist}

\DeclareMathOperator{\divg}{div}

\newcommand{\eps}{\varepsilon}

\theoremstyle{plain} \newtheorem{defi}{Definition}
\theoremstyle{plain} \newtheorem{rema}[defi]{Remark}
\theoremstyle{plain} \newtheorem{theo}[defi]{Theorem}
\theoremstyle{plain} \newtheorem{prop}[defi]{Proposition}
\theoremstyle{plain} \newtheorem{coro}[defi]{Corollary}
\theoremstyle{plain} \newtheorem{lemm}[defi]{Lemma}
\theoremstyle{plain} 
\theoremstyle{plain} 
\theoremstyle{plain} \newtheorem*{theo*}{Theorem}
\theoremstyle{plain} 
\theoremstyle{plain} 
\theoremstyle{plain} \newtheorem{clai}[defi]{Claim}

\numberwithin{equation}{section} % number equations according to the current section number
\numberwithin{defi}{section} % number theorems/etc according to the current section number

\allowdisplaybreaks

\title[Ancient gradient flows]{Ancient gradient flows of elliptic functionals and Morse index}
\date{}
\author{Kyeongsu Choi}
\author{Christos Mantoulidis}
\address{Korea Institute for Advanced Study (KIAS), 85 Hoegiro, Dongdaemun-gu, Seoul 02455, Republic of Korea}
\address{Brown University, Box 1917, 151 Thayer Street, Providence, RI 02912}

\begin{document}
	
\begin{abstract}
We study closed ancient solutions to gradient flows of elliptic functionals in Riemannian manifolds, including mean curvature flow and harmonic map heat flow.  Our work has various consequences. In all dimensions and codimensions, we classify ancient mean curvature flows in $\SS^n$ with low area: they are  steady or shrinking equatorial spheres. In the mean curvature flow case in $\SS^3$, we classify ancient flows with more relaxed area bounds: they are steady or shrinking equators or Clifford tori. In the embedded curve shortening case in $\SS^2$, we completely classify ancient flows of bounded length: they are steady or shrinking circles.
\end{abstract}

\maketitle

\input{intro}

\input{background}

\input{existence}

\input{uniqueness}

\input{mean-curvature-flow}

\appendix

\input{slow-examples}

\input{ode-lemma}

\input{regularity}

\end{document}

%% file: intro.tex
%!TEX root = main.tex

\section{Introduction} \label{sec:introduction}

\subsection{Mean curvature flow} \label{subsec:intro.mcf}

The mean curvature flow is a one-parameter family of submanifolds $\Sigma_t$ of a Riemannian manifold $(M, \overline{g})$ satisfying the evolution equation
\begin{equation} \label{eq:mcf.pde}
	\tfrac{\partial}{\partial t} x = \mathbf{H}(x, t), \; x \in \Sigma_t,
\end{equation}
where $\mathbf{H}(x, t)$ denotes the mean curvature vector of the $\Sigma_t$ at $x$, and which is the negative gradient of the area element of $\Sigma_t$. As a gradient flow of the area functional, the mean curvature flow describes a natural area minimizing process. Moreover, in Euclidean space $(M, \overline{g})= (\RR^n, dx_1^2 + \ldots, dx_n^2)$, the normalized mean curvature flow is a gradient flow of the Huisken density \cite{Huisken90}.

In this paper, we shall discuss ancient solutions of the mean curvature flow in Riemannian manifolds; that is, solutions existing for $t\in (-\infty,T)$. Since the mean curvature flow is a gradient flow, ancient solutions with finite energy are quite rare. Therefore, the classification of ancient solutions has been studied as a type of parabolic Liouville theory. 

There have been a number of important classification results for ancient mean curvature flows inside  Euclidean space, under suitable assumptions on the convexity or the entropy of the flow:
\begin{itemize}
	\item X.-J. Wang \cite{Wang11}  showed that a closed \emph{convex} ancient solution converges locally to a sphere or a cylinder after rescaling. Huisken--Sinestrari \cite{HuiskenSinestrari15} proved that a closed convex ancient solution with suitably \emph{pinched curvatures} must be the shrinking sphere. In the one dimensional case, Daskalopoulos--Hamilton--Sesum \cite{DaskalopoulosHamiltonSesum10} proved that the shrinking circle and the Angenent ovals are the only closed ancient solutions.
	\item Angenent--Daskalopoulos--Sesum \cite{AngenentDaskalopoulosSesum15, AngenentDaskalopoulosSesum18} showed that the only \emph{non-collapsed} closed \emph{2-convex} ancient solutions are the shrinking sphere and ancient ovals, which were constructed by White \cite{White03} and later by  Haslhofer--Hershkovits \cite{HaslhoferHershkovits16}. 
	 Bourni--Langford--Tinaglia \cite{BourniLangfordTinaglia} have constructed \emph{collapsed} examples of closed convex ancient flows.
	\item Brendle and the first author \cite{BrendleChoi19, BrendleChoi18} settled the uniqueness of \emph{non-collapsed} complete non-compact \emph{2-convex} ancient flows.
	\item The first author, Haslhofer and Hershkovits showed in \cite{ChoiHaslhoferHershkovits}, en route to proving Ilmanen's ``mean convex neighborhood'' conjecture, that a \emph{low entropy} ancient solution in $\RR^3$ must be one of the convex complete (or closed) non-collapsed ancient solutions, which were classified in \cite{AngenentDaskalopoulosSesum18} and \cite{BrendleChoi19}.
\end{itemize}

Much less is known about ancient solutions in Riemannian manifolds:
\begin{itemize}
	\item Huisken--Sinestrari \cite{HuiskenSinestrari15} showed that closed \emph{mean convex} and suitably \emph{curvature pinched} ancient solutions in $\SS^n$, $n \geq 3$, must be a shrinking spherical cap.
	\item Bryan--Louie \cite{BryanLouie16} showed that the only closed \emph{convex} ancient solutions in $\SS^2$ are shrinking circles. Bryan--Ivaki--Scheuer \cite{BryanIvakiScheuer} extended that conclusion   to \emph{convex fully nonlinear} flows in $\SS^n$, $n \geq 3$, including the mean curvature flow.
\end{itemize} 

In this work we study ancient mean curvature flows of closed submanifolds in Riemannian manifolds, as well as more general gradient flows of elliptic functionals. The goal, roughly, is to derive a sharp characterization of a large class of ancient flows as arising from the ``unstable manifold'' (i.e., the space of unstable directions for the area functional) of a given closed minimal submanifold. See Section \ref{subsec:intro.abstract} for the more abstract framework. 

Our work can be used to classify \emph{low area} ancient solutions in $\SS^n$ in arbitrary codimension:

\begin{theo} \label{theo:mcf.allard}
	There exists a $\delta = \delta(n) > 0$ such that if $(\Sigma_t)_{t \leq 0}$ is an ancient mean curvature flow of closed $m$-dimensional surfaces embedded in a round sphere $\SS^n$, with
	\begin{equation} \label{eq:mcf.allard.assumption}
	\lim_{t \to -\infty} \operatorname{Area}(\Sigma_t) < (1 + \delta) \operatorname{Area}(\SS^m),
	\end{equation}
	then $(\Sigma_t)_{t \leq 0}$ is a steady  or a shrinking equatorial $\SS^m$ one along one of $n-m$ orthogonal directions.
\end{theo}

\begin{rema} \label{rema:mcf.high.codim}
	Mean curvature flow has been studied primarily in codimension 1 as a result of the more subtle nature of singularity formation in high codimensions. On the other hand, the parabolic Liouville nature of  ancient gradient flows renders them quite rigid. As such, they can serve as a tangible stepping stone to a better understanding of high codimension mean curvature flows. Theorem \ref{theo:mcf.allard} is an example of a result for mean curvature flows that can be obtained just as easily in high codimension as in codimension 1. 
\end{rema}

This theorem has some interesting consequences. First, it recovers the classification of convex ancient mean curvature flows in $\SS^n$ in \cite{BryanIvakiScheuer, BryanLouie16}, since ancient convex solutions will satisfy \eqref{eq:mcf.allard.assumption}.

Second, it implies a complete classification of ancient embedded curve shortening flows in $\SS^2$ with bounded length:

\begin{coro} \label{coro:csf}
	Let $(\Gamma_t)_{t \leq 0}$ be an ancient curve shortening flow of embedded curves inside a round 2-sphere with
	\begin{equation} \label{eq:csf.length.assumption}
	\lim_{t \to -\infty} \operatorname{Length}(\Gamma_t) < \infty.
	\end{equation}
	Then $(\Gamma_t)_{t \leq 0}$ is a steady or a shrinking equator along circles of latitude.
\end{coro}

Huisken conjectures there exist ancient solutions that fill out $\SS^2$ as $t\to -\infty$, so one expects that assumption \eqref{eq:csf.length.assumption} is sharp.

\begin{rema} \label{rema:csf.rp2}
	One gets a classification of ancient embedded curve shortening flows with bounded length in $\RR \PP^2$ by lifting to $\SS^2$ and applying Corollary \ref{coro:csf}: they are steady equators and circles of latitude coming out of a multiplicity two equator. A similar proof shows that no nonsteady ancient embedded curve shortening flows with bounded length exist in flat tori or closed hyperbolic surfaces.
\end{rema}

Theorem \ref{theo:mcf.allard} can also be strengthened in $n=3$ dimensions due to the validity of the Willmore conjecture, proven by Marques--Neves \cite{MarquesNeves14}. Recall that the Clifford torus
\[ \{ (x, y, z, w) \in \RR^2 \times \RR^2 : x^2 + y^2 = z^2 + w^2 = \tfrac12 \} \subset \SS^3 \]
is a smoothly embedded minimal submanifold of $\SS^3$ with area $2\pi^2$. By the work of Marques--Neves \cite{MarquesNeves14}, this is the second smallest area among \emph{smooth} minimal surfaces, following the equatorial $\SS^2$ (area $4\pi$). We can show:

\begin{coro} \label{coro:mcf.s3}
	Let $(\Sigma_t)_{t \leq 0}$ be an ancient mean curvature flow of closed surfaces in a round $\SS^3$, with
	\begin{equation} \label{eq:mcf.s3.assumption}
		\lim_{t \to -\infty} \operatorname{Area}(\Sigma_t) < 2\pi^2 + \delta.
	\end{equation}
	If $\delta > 0$ is sufficiently small, then either:
	\begin{itemize}
		\item $\lim_{t \to -\infty} \operatorname{Area}(\Sigma_t) = 4\pi$, and $(\Sigma_t)_{t \leq 0}$ is a steady or shrinking equator along spheres of latitude; or,
		\item $\lim_{t \to -\infty} \operatorname{Area}(\Sigma_t) = 2\pi^2$, and $(\Sigma_t)_{t \leq 0}$ is a steady or shrinking Clifford torus along one of its 5 linearly unstable directions. 
	\end{itemize}
\end{coro}

Recall also that the number of linearly unstable directions (the ``Morse index'') of a Clifford torus was computed by Urbano \cite{Urbano90} to be 5. See \cite{MarquesNeves14} for a geometric interpretation of these 5 unstable directions.

We now summarize our tools, which should be interesting in their own right. First, we prove that the Morse index of a minimal submanifold gives rise to a family of exponentially decaying ancient mean curvature flows:

\begin{theo}[cf. Theorem \ref{theo:existence}]  \label{theo:intro.mcf.existence}
	Let $S$ be a closed, smoothly embedded minimal submanifold in a Riemannian manifold $(M, \overline{g})$ with Morse index $I \in \NN$. Then there exists an $I$-parameter family of ancient mean curvature flows on $(-\infty, 0]$ that are uniquely determined by their trace at time $t=0$ and converge exponentially quickly to $S$ as $t \to -\infty$.
\end{theo}

Only few non-convex (or nonpositively curved) ancient solutions to geometric flows have been previously discovered; see, e.g., the ancient Yamabe flow from two spheres \cite{DaskalopoulosdelPinoSesum}. Theorem \ref{theo:intro.mcf.existence} shows the existence of infinitely many non-steady non-convex ancient solutions.

Second, we prove a sharp characterization of ancient flows; if a flow decays as $t \to -\infty$ in an ``integrable'' ($L^1$) sense, then it is one of the flows that was generated by the Morse index.

\begin{theo}[cf. Theorem \ref{theo:mcf.characterization}]  \label{theo:intro.mcf.characterization}
	Let $S$ be a closed, smoothly embedded minimal submanifold of a Riemannian manifold $(M, \overline{g})$. There exists an $\eps > 0$ such that if $(\Sigma_t)_{t \leq 0}$ is an ancient mean curvature flow which stays uniformly $\eps$-close to $S$ in the sense of measures, and
	\begin{equation} \label{eq:intro.mcf.characterization.assumption}
		\int_{-\infty}^0 \dist_{\overline{g}}(\Sigma_t, S) \, dt < \infty,
	\end{equation}
	then there exists $\tau \geq 0$ so that $(\Sigma_{t-\tau})_{t \leq 0}$ is one of the flows from Theorem \ref{theo:intro.mcf.existence}.
\end{theo}

Assumption \eqref{eq:intro.mcf.characterization.assumption} is key for the conclusion. Indeed, in Appendix \ref{sec:slow.convergence} we construct examples of flows which are not generated by a negative eigenfunction and which decay arbitrarily slowly as $t \to -\infty$. 

We also give a sufficient geometric condition which guarantees the decay needed for \eqref{eq:intro.mcf.characterization.assumption}. Indeed, we show that ancient flows that remain suitably close to a so-called ``integrable critical point'' (see Definition \ref{defi:integrable}), will converge exponentially quickly, as $t \to -\infty$, to a (possibly different) critical minimal submanifold. This notion of integrability and its implication on rates of convergence was pioneered by Allard and Almgren  \cite{AllardAlmgren81} in their study of tangent cones of minimal surfaces with isolated singularities.%, in the elliptic setting, and later by Simon \cite{Simon84} in the forward-in-time parabolic setting. 

\begin{prop}[cf. Proposition \ref{prop:mcf.integrable}]  \label{prop:intro.mcf.integrable}
	Let $S$ be an integrable, closed, smoothly embedded minimal submanifold of a Riemannian manifold $(M, \overline{g})$. There exist $\eps, c, \kappa > 0$ such that if $(\Sigma_t)_{t \leq 0}$ is an ancient mean curvature flow which stays uniformly $\eps$-close to $S$ in the sense of measures, and
	\begin{equation} \label{eq:intro.mcf.integrable.assumption}
		\lim_{t \to -\infty} \operatorname{Area}_{\overline{g}}(\Sigma_t) \leq \operatorname{Area}_{\overline{g}}(S),
	\end{equation}
	then $\Sigma_t$ is $c e^{\kappa t}$-close, in the $C^{2,\theta}$ sense, to a (possibly different) fixed, closed, smoothly embedded minimal submanifold.
\end{prop}

Two important cases that automatically guarantee \eqref{eq:intro.mcf.integrable.assumption} are:
\begin{enumerate}
	\item when $S$ is nondegenerate (i.e., its linearization has no eigenvalues equal to zero), or
	\item when the ambient Riemannian metric is real analytic.
\end{enumerate}
See Remark \ref{rema:integrable.sufficient}.

\subsection{General theory} \label{subsec:intro.abstract}

Let $(\Sigma, g)$ be a closed Riemannian manifold and $V \to \Sigma$ be a Euclidean vector bundle. We consider ancient solutions to gradient flows for functionals of the form
\begin{equation} \label{eq:elliptic.functional}
	\cA(f) := \int_\Sigma A(x, f(x), \nabla_g f(x)) \, d\mu_g(x),
\end{equation}
whose arguments are sections $f$ of the bundle $V$ and whose integrand $A(x, z, q)$ is such that:
\begin{enumerate}
	\item $A(x, z, q)$ is a smooth real-valued function of $(x, z, q)$, $x \in \Sigma$, $z \in V_x$, $q \in T_x \Sigma \otimes V_x$;
	\item $A(x, z, q)$ satisfies the Legendre--Hadamard ellipticity condition 
		\begin{equation} \label{eq:elliptic.functional.convexity}
			\left[ \tfrac{d^2}{ds^2} A(x, 0, s (\tau \otimes v)) \right]_{s=0} \geq c |\tau|^2 |v|^2,
		\end{equation}
		for $c > 0$ independent of $x \in \Sigma$, $\tau \in T_x \Sigma$, $v \in V_x$.
\end{enumerate}
The negative $L^2$ gradient of $\cA(f)$, denoted $\cH(f)$, is determined by the pairing
\begin{equation} \label{eq:negative.gradient}
	\langle \cH(f), \zeta \rangle_{L^2(\Sigma)} = - \left[ \tfrac{d}{ds} \cA(f + s \zeta) \right]_{s=0}, \; \forall \zeta \in C^\infty(\Sigma; V).	
\end{equation}
A ``gradient flow'' of $\cA$ is an  evolution equation
\begin{equation} \label{eq:negative.gradient.flow}
	\tfrac{\partial}{\partial t} u = \cH(u).
\end{equation}
A solution $u$ of \eqref{eq:negative.gradient.flow} is called \emph{ancient} if its time domain contains an interval of the form $(-\infty, T)$, $T \in \RR$. In this paper we are interested in smooth solutions of \eqref{eq:negative.gradient.flow}.

Our main results are described below. We refer the reader to Sections \ref{sec:notation},  \ref{sec:existence}, \ref{sec:uniqueness} for all relevant definitions and precise statements.

\begin{theo}[cf. Theorem \ref{theo:existence}, Corollary \ref{coro:existence}] \label{theo:intro.existence}
	Let $0$ be a critical point of $\cA$ with Morse index $I \in \NN$. There exists an $I$-parameter family of ancient solutions to \eqref{eq:negative.gradient.flow}, which are uniquely determined by their trace at $t=0$ and which converge to $0$ exponentially as $t \to -\infty$. The space of their traces at $t=0$ is tangent to the $I$-dimensional space of negative eigenfunctions.
\end{theo}

\begin{theo}[cf. Theorem \ref{theo:characterization}] \label{theo:intro.characterization}
	Let $0$ be a critical point of $\cA$, $\theta \in (0, 1)$, $C_0 > 0$. There exists $\eps > 0$ such that if $u : \Sigma \to V$ is a smooth ancient solution of \eqref{eq:negative.gradient.flow} with spatial $C^1$ norm bounded by $\eps$, parabolic $C^{1,\theta}$ norm bounded by $C_0$, and finite spacetime $L^1$ norm, then $u$ belongs to the space of solutions from Theorem \ref{theo:intro.existence}.
\end{theo}

\begin{prop}[cf. Proposition \ref{prop:integrable}] \label{prop:intro.integrable}
	Let $0$ be an integrable critical point of $\cA$, and $\theta \in (0, 1)$. There exist $\eps, c, \kappa > 0$ such that if $u$ is a smooth ancient solution of \eqref{eq:negative.gradient.flow} with parabolic $C^{1,\theta}$ norm bounded by $\eps$ and
	\begin{equation} \label{eq:intro.integrable.assumption}
		\lim_{t \to -\infty} \cA(u(\cdot, t)) \leq \cA(0),
	\end{equation}
	then $u$ is $c e^{\kappa t}$-close in the parabolic $C^{1,\theta}$ sense to a fixed (but possibly different) critical point of $\cA$.
\end{prop}

See Remark \ref{rema:integrable.sufficient} for natural sufficient conditions that guarantee the validity of \eqref{eq:intro.integrable.assumption}.

\begin{rema} \label{rema:mcf.vs.abstract}
	A subtle remark is in order regarding whether the results of this section immediately imply those of Section \ref{subsec:intro.mcf}. One could hope to immediately recover the results of Section \ref{subsec:intro.mcf} by taking $V$ to be the normal bundle $NS$ of $\Sigma \subset (M, \overline{g})$ and defining the elliptic functional $\cA$ as the area of the graphical submanifold induced by a map $f : S \to NS$. While this is an admissible functional (see, e.g., \cite{Simon83}), we point out that its  gradient flow is a ``nonparametric'' gradient flow, so it differs from the classical mean curvature flow considered in Section \ref{subsec:intro.mcf}, which is a ``parametric'' gradient flow for the area functional of embedded submanifolds. This detail, unfortunately, interferes with the divergence structure of the evolution equation \eqref{eq:negative.gradient.flow}. With this in mind, we have sought to exploit the divergence structure as little as possible in order for our proofs to carry over, with only minor modifications, to the classical mean curvature flow setting in Section \ref{subsec:intro.mcf}. We discuss these modifications in very specific terms in Section \ref{sec:mcf}, where we give the proofs of the results announced in Section \ref{subsec:intro.mcf},
\end{rema}

\subsection{Harmonic map heat flow} \label{subsec:intro.hmhf}

Let $(M, g)$ be a closed Riemannian manifold (the ``domain'') and $(N, h)$ be another Riemannian manifold (the ``target''). The harmonic map heat flow is the gradient flow of the Dirichlet energy functional
\[ E(f) := \tfrac12 \int_M \Vert df \Vert^2 \, d\mu_g, \; f \in C^1(M; N). \]
Namely, it is the flow
\[ \tfrac{\partial}{\partial t} f = \tau(f(\cdot, t)), \]
where $\tau$ denotes the negative $L^2$ gradient of the Dirichlet energy functional. Our results, namely Theorems \ref{theo:intro.existence}, \ref{theo:intro.characterization}, and Proposition \ref{prop:intro.integrable} apply to harmonic map heat flows modulo the same minor modifications that had to be carried out for mean curvature flow; namely, modifications to go from the ``parametric'' gradient flow (the harmonic map heat flow) to the ``nonparametric'' gradient flows discussed in Section \ref{subsec:intro.abstract}.

\subsection*{Outline of paper and some motivation}   

In Section \ref{sec:notation} we set up our notation and relevant necessary background. In Sections \ref{sec:existence}, \ref{sec:uniqueness}, we show the existence and uniqueness of ancient gradient flows within the class of flows that originate, with certain $L^1$ control, out of a critical point. In Section \ref{sec:mcf} we extend our results to ancient mean curvature flows. In Appendix \ref{sec:slow.convergence} we discuss examples of flows with slow convergence which therefore are not meant to meet our characterization. In Appendix \ref{sec:mz.ode} we discuss an extension of an ODE lemma due to Merle--Zaag \cite{MerleZaag98} that we need. In Appendix \ref{sec:regularity} we discuss the form of Schauder estimates we need for our linear parabolic systems. 

Our study of ancient gradient flows requires a few ideas that are familiar to the experts of two neighboring fields:
\begin{itemize}
	\item the study of minimal surfaces with isolated singularities;
	\item the forward-time study of uniqueness of tangent flows for mean curvature flow at the first singular time.
\end{itemize}
Namely, we use the notion of integrable critical points for Proposition \ref{prop:intro.integrable} and the \L{}ojasiewicz--Simon inequality for Theorem \ref{theo:mcf.allard}. For context, see the pioneering works of Allard--Almgren \cite{AllardAlmgren81} and Simon \cite{Simon83}. The \L{}ojasiewicz--Simon inequality has found spectacular success in the study of singularities in mean curvature flow: novel variants were used by Schulze \cite{Schulze14}, Colding--Minicozzi \cite{ColdingMinicozzi15}, and Chodosh--Schulze \cite{ChodoshSchulze} to prove uniqueness of certain ``multiplicity one'' tangent flows. The ``dynamical'' study of singularities in the recent work of Colding--Minicozzi \cite{ColdingMinicozzi18a, ColdingMinicozzi18b} is also reminiscent of some aspects of Theorem \ref{theo:mcf.allard}. 

Theorem \ref{theo:intro.characterization} follows a different set of ideas. Key is a Cacciopoli type inequality, \eqref{eq:characterization.w12.l2}, which is deeply connected to the \emph{ancient} and the \emph{gradient} nature of the flow. For context, see Angenent--Daskalopoulos--Sesum \cite[Lemma 4.12] {AngenentDaskalopoulosSesum15}. The Caccioppoli inequality lets us estimate the $C^{2,\theta}$ decay of our flow in terms of its $L^2$ energy, which relates more naturally to the gradient nature of the flow. Indeed, we  decompose the $L^2$ norm into the stable, neutral, and unstable components, and directly study the dynamics of these components by building on an ODE result originally due to Merle--Zaag \cite{MerleZaag98}; see Lemma \ref{lemm:mz.ode}.

\subsection*{Acknowledgments}

KC was supported by the National Science Foundation under grant DMS-1811267. CM was supported by the National Science Foundational under grant DMS-1905165. We are grateful to F. C. Marques and J. Bernstein for suggesting to us that Corollary \ref{coro:mcf.s3} (which did not appear on the first version of the paper) follows from our proof of Theorem \ref{theo:mcf.allard} and the (now proven) Willmore conjecture. We are grateful to the journal's referee for their recommendations. We would like to thank  B. Choi for pointing out a point that had to be clarified, F. Schulze, O. Hershkovits, C. Mooney, and N. Edelen for insightful conversations, and T. Colding, B. Minicozzi, M. Langford, T. Bourni, M. Ivaki, Y. Sire, and A. Payne for their interest.

%% file: background.tex
%!TEX root = main.tex

\section{Background and notation} \label{sec:notation}

\subsection{Functional spaces} 

Let $(\Sigma, g)$ be a closed Riemannian manifold and $V \to \Sigma$ be a Euclidean vector bundle. Let $\Omega \subset \Sigma \times \RR$, $\theta \in (0,1]$. For $u : \Omega \to V$ we define:
\[ [u]_{C^\theta_P(\Omega; V)} := \sup \left\{ \frac{d_V(u(p,t), u(q, s))}{d_\Sigma(p,q)^\theta + |t-s|^{\theta/2}} : (p, t), (q,s) \in \Omega, \; (p,t) \neq (q,s) \right\}, \]
and for $k \in \NN$:
\[ \Vert u \Vert_{C^{k,\theta}_P(\Omega; V)} := \sum_{i+2j \leq k} \sup_\Omega \Vert \nabla_x^i \nabla_t^j u \Vert + \sum_{i+2j=k} [\nabla_x^i \nabla_t^j u]_{C^\theta_P(\Omega; V)}. \]
The corresponding parabolic H\"older spaces are $C^{k,\theta}_P(\Omega; V)$. 

Now suppose $\Omega \subset \Sigma$. Without the subscript $P$, $[u]_{C^\theta(\Omega; V)}$, $\Vert u \Vert_{C^{k,\theta}(\Omega; V)}$ refer to the standard seminorm and norm of the Banach space $C^{k,\theta}(\Omega; V)$. 

Finally, when $\Omega = \Sigma \times \RR_-$, we will need to consider spaces of functions with controlled exponential decay. For $\delta > 0$, define
\begin{equation} \label{eq:exponential.decay.norm}
	\Vert u \Vert_{C^{k,\theta,\delta}_P(\Sigma \times \RR_-; V)} := \sup_{t \in \RR_-} \left[ e^{-\delta t} \Vert u \Vert_{C^{k,\theta}_P(\Sigma \times [t-1,t]; V)} \right].
\end{equation}
The vector space
\[ C^{k,\theta,\delta}_P(\Sigma \times \RR_-; V) := \{ u \in C^{k,\theta}_P(\Sigma \times \RR_-; V) : \Vert u \Vert_{C^{k,\theta,\delta}_P(\Sigma \times \RR_-; V)} < \infty \} \]
is evidently a Banach space when endowed with $\Vert \cdot \Vert_{C^{k,\theta,\delta}_P(\Sigma \times \RR_-; V)}$.

\subsection{Space of critical points} \label{sec:notation.critical.points}

We will be actively interested in the space of critical points
with small $C^{k,\theta}$ norm:
\begin{equation} \label{eq:c.k.alpha.ball.critical}
	\cM^{k,\theta}(\delta) := \{ f : \cH(f) = 0, \; \Vert f \Vert_{C^{k,\theta}(\Sigma; V)} < \delta \},
\end{equation}
particularly when $f=0$ is itself a critical point, which we will assume throughout this paper. From,  \eqref{eq:negative.gradient} we find that
\begin{equation} \label{eq:negative.gradient.divergence.form}
	\cH(f) = \divg_g \big[ \nabla_{q} A(x, f, \nabla_g f) \big] - \nabla_{z} A(x, f, \nabla_g f).
\end{equation}

We interpret $\cH(f) = 0$ as a weak second order divergence-form system as in \eqref{eq:negative.gradient.divergence.form}. Schauder theory for elliptic systems \cite{Simon97} implies that $\cM^{1,\theta}(\delta)$, $\theta \in (0, 1)$, already captures all solutions near the origin as long as one suitably adjusts $\delta$.

\begin{rema} \label{rema:critical.point.regularity}
	When $V$ is a line bundle, elliptic De Giorgi--Nash--Moser theory \cite[Chapter 8]{GilbargTrudinger01} allows us to use  $\cM^{1,0}(\delta')$ instead of $\cM^{1,\theta}(\delta)$.
\end{rema}

The linearization of $\cH(f)$ at $f=0$ will play an important role in our work, so let us define:
\begin{equation} \label{eq:jacobi.operator}
	Lf := \left[ \tfrac{d}{ds} \cH(sf) \right]_{s=0}.
\end{equation}
An elementary computation involving \eqref{eq:elliptic.functional.convexity}, \eqref{eq:negative.gradient.divergence.form} shows that
\[ Lf = \divg_g \big[ \langle \nabla^2_{q} A(x, 0, 0), \nabla_g f \rangle_g \big] + (\divg_g \nabla_{q} \nabla_{z} A(x, 0, 0) - \nabla_z^2 A(x, 0, 0)) f \]
is a uniformly elliptic self-adjoint divergence form operator. We will denote the eigenvalues of $-L$ as
\begin{equation} \label{eq:jacobi.operator.eigenvalues}
	\lambda_1 < \lambda_2 \leq \ldots \leq \lambda_I < \lambda_{I+1} = \ldots = \lambda_{I+K} = 0 < \lambda_{I+K+1} \leq \ldots,
\end{equation}
repeated according to their multiplicity; note that $\lim_j \lambda_j = \infty$ \cite[Chapter 5]{GilbargTrudinger01}. Here, $I = \ind(L)$ is the ``Morse index'' of $L$, and $K = \nul(L)$ is the ``nullity'' of $L$. We also fix once and for all an $L^2$ orthonormal sequence of corresponding eigenfunctions $\varphi_j : \Sigma \to V$:
\begin{itemize}
	\item $\varphi_1, \ldots, \varphi_I$ are called ``unstable modes'',
	\item $\varphi_{I+1}, \ldots, \varphi_{I+K}$ are called ``neutral modes'' or ``Jacobi fields'',
	\item $\varphi_{I+K+1}, \varphi_{I+K+2}, \ldots$ are called ``stable modes''.
\end{itemize}
We consider auxiliary operators:
\begin{equation} \label{eq:iota.minus}
	\iota_- : \RR^I \to L^2(\Sigma \times \RR_-; V), \; \iota_-(\bm{a}) := \sum_{j=1}^I a_j e^{-\lambda_j t} \varphi_j,
\end{equation}
\begin{multline} \label{eq:pi.minus}
	\Pi_- : L^2(\Sigma; V) \to L^2(\Sigma; V), \\
	\Pi_- \varphi := \iota_-(\langle \varphi, \varphi_1 \rangle_{L^2(\Sigma; V)}, \ldots, \langle \varphi, \varphi_I \rangle_{L^2(\Sigma; V)})(\cdot, 0),
\end{multline}
\begin{equation} \label{eq:iota.0}
	\iota_0 : \RR^K \to L^2(\Sigma; V), \; \iota_0(a_1, \ldots, a_K) := \sum_{\ell=1}^K a_\ell \varphi_{I+\ell},
\end{equation}
\begin{multline} \label{eq:pi.0}
	\Pi_0: L^2(\Sigma; V) \to L^2(\Sigma; V), \\
	\Pi_0 \varphi := \iota_0(\langle \varphi, \varphi_{I+1} \rangle_{L^2(\Sigma; V)}, \ldots, \langle \varphi, \varphi_{I+K} \rangle_{L^2(\Sigma; V)}).
\end{multline}

We now briefly recall the structure result for $\cM^{2,\theta}(\delta)$ in \cite[Section 2]{Simon83}. It is occasionally convenient to rewrite 
\begin{equation} \label{eq:critical.point.operator.with.errors}
\cH(f) = Lf + \langle \mathscr{N}(x, f, \nabla_g f), \nabla_g^2 f \rangle_{g} + \mathscr{Q}(x, f, \nabla_g f),
\end{equation} 
where $L$ is as above; $\mathscr{N}(x, z, q)$ is a smooth symmetric bilinear form mapping into $V$ satisfying
\begin{equation} \label{eq:critical.point.equation.error.i}
	(\Vert z \Vert + \Vert q \Vert)^{\min\{-1+j+k,0\}} \Vert \nabla_x^i \nabla_z^j \nabla_{q}^k \mathscr{N}(x, z, q) \Vert \leq c, \; i, j, j \geq 0;
\end{equation}
and $\mathscr{Q}(x, z, q)$ is a smooth $V$-valued function satisfying
\begin{equation} \label{eq:critical.point.equation.error.ii}
	(\Vert z \Vert + \Vert q \Vert)^{\min\{-2+j+k,0\}} \Vert \nabla_x^i \nabla_z^j \nabla_{q}^k \mathscr{Q}(x, z, q) \Vert \leq c, \; i, j, k \geq 0.
\end{equation}
Adding $\Pi_0$ from \eqref{eq:pi.0} to both sides of $\cH(f) = 0$, and recalling  \eqref{eq:critical.point.operator.with.errors}, the critical point equation is equivalent to
\begin{equation} \label{eq:critical.point.equation.with.projection}
	Lf + \langle \mathscr{N}(x, f, \nabla_g f), \nabla_g^2 f \rangle_{g} + \mathscr{Q}(x, f, \nabla_g f) + \Pi_0 f = \Pi_0 f.
\end{equation}
By the invertibility $L + \Pi_0$, the implicit function theorem on Banach spaces implies that there exist neighborhoods $W_1$, $W_2$ of $0$ in $C^{2,\theta}(\Sigma; V)$, $C^{0,\theta}(\Sigma; V)$, and a diffeomorphism $\Psi : W_2 \to W_1$ such that
\begin{align} 
	(L + \langle \mathscr{N}, \nabla_g^2 \rangle_g + \mathscr{Q} + \Pi_0) \circ \Psi = \operatorname{Id}_{W_2}, \label{eq:critical.point.implicit.i} \\
	\Psi \circ (L + \langle \mathscr{N}, \nabla_g^2 \rangle_g + \mathscr{Q} + \Pi_0) = \operatorname{Id}_{W_1}. \label{eq:critical.point.implicit.ii}
\end{align}
Set $U := \{ \bm{a} \in \RR^K : \iota_0(\bm{a}) \in W_2 \} \subset \RR^K$, and consider the finite dimensional reduction $\cA_{\operatorname{fin}} : U \to \RR$,
\[ \cA_{\operatorname{fin}}(\bm{a}) := \cA(\Psi(\iota_0(\bm{a}))). \]
For $f$ with $\Pi_0 f \in W_2$, \eqref{eq:critical.point.equation.with.projection} is equivalent to $f = (\Psi \circ \iota_0)(\bm{a})$ with $\bm{a} \in U$, $\nabla \cA_{\operatorname{fin}}(\bm{a}) = \bm{0}$. Shrinking $W_1$, we conclude that for small $\delta > 0$, 
\begin{equation} \label{eq:critical.point.space}
	\cM^{2,\theta}(\delta) = \{ (\Psi \circ \iota_0)(\bm{a}) : \bm{a} \in U, \; \nabla \cA_{\operatorname{fin}}(\bm{a}) = \bm{0} \}
\end{equation}
for some open neighborhood $U$ of $\bm{0} \in \RR^K$.  The same representation will also hold true for $\cM^{1,\theta}(\delta')$, for a smaller $\delta' > 0$, by elliptic Schauder theory \cite[Chapter 6]{GilbargTrudinger01}; see also Remark \ref{rema:critical.point.regularity}.

%% file: existence.tex
%!TEX root = main.tex

\section{Existence of ancient flows} \label{sec:existence}

In this section we construct ancient flows converging exponentially quickly to arbitrary unstable critical points along their unstable eigenspaces, in the spirit of an unstable manifold theorem. Results of this flavor are true for nonlinear parabolic PDEs in various settings (see \cite[Chapter 9]{Lunardi95}), but we opt for a short self-contained proof modeled on an elliptic result of Caffarelli--Hardt--Simon \cite{CaffarelliHardtSimon84}. As a side consequence of our contraction mapping technique we naturally get a uniqueness result within exponentially decaying flows, but we will sharpen this uniqueness substantially in Section \ref{sec:uniqueness}. 

We start by considering the inhomogeneous linear PDE
\begin{equation} \label{eq:linear.pde.inhomogeneous}
\tfrac{\partial}{\partial t} u = Lu + h(x, t), \; (x, t) \in \Sigma \times \RR_-,
\end{equation}
where $h : \Sigma \times \RR_- \to V$ is some given smooth function. It is well known that solutions of \eqref{eq:linear.pde.inhomogeneous} can be expressed as
\begin{equation} \label{eq:linear.fourier.series}
	u(x, t) = \sum_{j=1}^\infty u_j(t) \varphi_j(x),
\end{equation}
and the $u_j$ are, formally, solutions of $u_j'(t) = -\lambda_j u_j(t) + h_j(t)$, where $h(x, t) = \sum_{j=1}^\infty h_j(t) \varphi_j(x)$. 

\begin{lemm} \label{lemm:linear.existence.L2}
	Suppose that $\delta > 0$ is such that
	\[ \int_{-\infty}^0 \left| e^{-\delta t} \Vert h(\cdot, t) \Vert_{L^2(\Sigma; V)} \right|^2 \, dt < \infty. \]
	Fix $\bm{a} \in \RR^I$. There exists a unique solution $u$ of \eqref{eq:linear.pde.inhomogeneous} such that
	\[ \Pi_-(u(\cdot, 0)) = \iota_-(\bm{a})(\cdot, 0), \]
	\[ \int_{-\infty}^0 \left| e^{-\delta' t} \Vert u(\cdot, t) \Vert_{L^2(\Sigma; V)} \right|^2 \, dt < \infty \]
	for some $0 < \delta' < \min\{\delta,-\lambda_I\}$. It is given by the series in  \eqref{eq:linear.fourier.series} with 
	\begin{equation} \label{eq:linear.coefficient.ansatz.negative}
	u_j(t) := a_j e^{-\lambda_j t} - \int_t^0 e^{\lambda_j(s-t)} h_j(s) \, ds, \; j = 1, 2, \ldots, I,
	\end{equation}
	\begin{equation} \label{eq:linear.coefficient.ansatz.positive}
	u_j(t) := \int_{-\infty}^t e^{\lambda_j(s-t)} h_j(s) \, ds, \; j = I+1, I+2, \ldots.
	\end{equation}
	For every $t \leq 0$, and $0 < \delta' < \min\{\delta, -\lambda_I\}$,
	% with strict inequality on the right when $\delta = -\lambda_I$,
	\begin{equation} \label{eq:linear.pde.bound.L2}
	e^{-\delta' t} \Vert u(\cdot, t) - \iota_-(\bm{a}) \Vert_{L^2(\Sigma; V)} \leq c \left[ \int_{-\infty}^0 \left| e^{-\delta \tau} \Vert h(\cdot, \tau) \Vert_{L^2(\Sigma; V)} \right|^2 \, d\tau \right]^{1/2},
	\end{equation}
	for some $c = c(\delta, \delta', \lambda_I) > 0$. 
\end{lemm}
\begin{proof}
	This is a straightforward computation given \eqref{eq:linear.coefficient.ansatz.negative}, \eqref{eq:linear.coefficient.ansatz.positive}.
\end{proof}

Schauder theory for linear parabolic equations implies:

\begin{coro} \label{coro:linear.existence.holder}
	Suppose that $h \in  C^{0,\theta,\delta}_P(\Sigma \times \RR_-)$ for $\theta \in (0, 1)$, $\delta > 0$. The solution in Lemma \ref{lemm:linear.existence.L2} satisfies, for every $0 < \delta'' < \min\{\delta, -\lambda_I\}$,
	\begin{equation} \label{eq:linear.pde.bound.holder}
		\Vert u - \iota_-(\bm{a})  \Vert_{C^{2,\theta,\delta''}_P(\Sigma \times \RR_-; V)} \leq c \Vert h \Vert_{C^{0,\theta,\delta}_P(\Sigma \times \RR_-; V)},
	\end{equation}
	for some $c = c(\theta, \delta, \delta'', \lambda_I) > 0$.
\end{coro}

We now turn to the construction of solutions $u : \Sigma \times \RR_- \to V$ to:
\begin{equation} \label{eq:existence.nonlinear}
	\tfrac{\partial}{\partial t} u = Lu + \langle \mathscr{N}(x, u, \nabla_g u), \nabla^2_g u \rangle_g + \mathscr{Q}(x, u, \nabla_g u),
\end{equation}
where $\mathscr{N}$, $\mathscr{Q}$ are as in \eqref{eq:critical.point.equation.error.i}, \eqref{eq:critical.point.equation.error.ii}. Note that ancient solutions of \eqref{eq:negative.gradient.flow} are precisely of this form because of  \eqref{eq:critical.point.operator.with.errors}.

\begin{theo} \label{theo:existence}
	Fix $\delta_0 \in (0, -\lambda_I)$.  There is a $\mu_0 > 0$ such that for any $\mu \geq \mu_0$, $\bm{a} \in B_\eta(\bm{0}) \subset \RR^I$, with $\eta$ depending on $\mu$, there is a unique solution  $\mathscr{S}(\bm{a}) : \Sigma \times \RR_- \to V$ of \eqref{eq:existence.nonlinear} satisfying 
	\begin{equation} \label{eq:nondegenerate.solution.operator}
		\Vert \mathscr{S}(\bm{a}) - \iota_-(\bm{a}) \Vert_{C^{2,\theta,\delta_0}_P(\Sigma \times \RR_-; V)} \leq \mu |\bm{a}|^2, \; \Pi_- \big[ \mathscr{S}(\bm{a})(\cdot, 0) \big] = \iota_-(\bm{a})(\cdot, 0).
	\end{equation}
\end{theo}
\begin{proof}
	The space
	\[ \cC(\bm{a}) := \{ u \in C^{2,\theta,\delta_0}_P(\Sigma \times \RR_-; V) : \Pi_- (u(\cdot, 0)) = \iota_-(\bm{a})(\cdot, 0) \} \} \]
	is a closed subspace of $C^{2,\theta,\delta_0}_P(\Sigma \times \RR_-; V)$, so it is also a Banach space. For brevity, we will write $\Vert \cdot \Vert$ for $\Vert \cdot \Vert_{C^{2,\theta,\delta_0}_P(\Sigma \times \RR_-; V)}$ in this proof.
	
	For $u \in \cC[\bm{a}] \cap C^\infty(\Sigma \times \RR_-; V)$, define $\mathscr{S}(u; \bm{a})$ to be a solution in $\cC[\bm{a}]$ of
	\[ (\tfrac{\partial}{\partial t} - L) \mathscr{S}(u; \bm{a}) = \langle \mathscr{N}(x, u, \nabla_g u), \nabla_g^2 u \rangle_g + \mathscr{Q}(x, u, \nabla_g u), \; t \leq 0. \]
	Existence and uniqueness hold by Corollary \ref{coro:linear.existence.holder}, which applies with $\delta = 2 \delta_0$ in view of \eqref{eq:critical.point.equation.error.i}-\eqref{eq:critical.point.equation.error.ii} and shows that, for some $c > 0$,
	\begin{align} 
	\Vert \mathscr{S}(u;\bm{a}) - \iota_-(\bm{a}) \Vert & \leq c \Vert u \Vert^2, \label{eq:existence.estimate.i} \\
	\Vert \mathscr{S}(u;\bm{a}) - \mathscr{S}(u'; \bm{a}) \Vert & \leq c (\Vert u \Vert + \Vert u' \Vert) \Vert u - u' \Vert, \label{eq:existence.estimate.ii}
	\end{align}
	By \eqref{eq:existence.estimate.ii}, $\mathscr{S}(\cdot; \bm{a})$ extends to a  $C^1$ map of $\cC[\bm{a}]$. By \eqref{eq:existence.estimate.i},   \eqref{eq:existence.estimate.ii}, $\mathscr{S}(\cdot; \bm{a})$ is a contraction mapping of the convex subset $\{ u \in \cC[\bm{a}] : \Vert u - \iota_-(\bm{a}) \Vert \leq \mu |\bm{a}|^2 \}$, provided $\mu > 2c$ and $\eta$ is small depending on $\mu$, $c$. The result follows from the contraction mapping principle.
\end{proof}

As an immediate corollary of \eqref{eq:existence.estimate.i}-\eqref{eq:existence.estimate.ii} we get:

\begin{coro} \label{coro:existence}
	The mapping $\mathscr{S} : B_\eta(\bm{0}) \to C^{2,\theta,\delta_0}_P(\Sigma \times \RR_-; V)$ of Theorem \ref{theo:existence} satisfies  $\mathscr{S}(\bm{0}) = 0$ and $\left[ \tfrac{d}{ds} \mathscr{S}(s\bm{a}) \right]_{s=0} = \iota_-(\bm{a})$, $\forall \bm{a} \in \RR^I$.
\end{coro}

In other words, $\mathscr{S}$ can be viewed as parametrizing the ``unstable manifold'' that corresponds to the critical point at the origin, which is tangent to the subspace of eigenfunctions of $L$ with negative eigenvalues. It is not hard to see that $\mathscr{S}$ is a smooth Banach functional.

%% file: uniqueness.tex
%!TEX root = main.tex

\section{Uniqueness} \label{sec:uniqueness}

Our main theorem in this section is:

\begin{theo} \label{theo:characterization}
	Fix $\theta \in (0, 1)$, $C_0 > 0$. There exists $\eps > 0$ such that if $u : \Sigma \times \RR_- \to V$ is a smooth solution of \eqref{eq:negative.gradient.flow} satisfying
	\begin{equation} \label{eq:characterization.assumption.c1theta}
		\Vert u \Vert_{C^{1,\theta}_P(\Sigma \times \RR_-; V)} \leq C_0,
	\end{equation}
	\begin{equation} \label{eq:characterization.assumption.small}
		\Vert u(\cdot, t) \Vert_{C^1(\Sigma; V)} < \eps, \; \forall t \leq 0,
	\end{equation}
	and
	\begin{equation} \label{eq:characterization.assumption.L1}
		\Vert u \Vert_{L^1(\Sigma \times \RR_-; V)} < \infty,
	\end{equation}
	then there exists $\tau \geq 0$ and $\bm{a} \in \RR^I$ such that $u(x, t - \tau)$ coincides with $\mathscr{S}(\bm{a})$ from Theorem \ref{theo:existence}.
\end{theo}

Let us spend a moment to contrast Theorem \ref{theo:characterization} to generic center-unstable manifold type results; see \cite[Chapter 9.2]{Lunardi95}, \cite{DaPratoLunardi88}. The latter ascertain that ancient flows near critical points decompose into a ``slow'' neutral component and a ``fast'' unstable component (exponentially decaying with speed $e^{\alpha t}$, $\alpha < -\lambda_I$). We prove a finer result by exploiting the quasilinear gradient flow structure. Iterating a lemma of Merle--Zaag (Appendix \ref{sec:mz.ode}), we study the dynamics across all eigenspace projections. We show that an ancient flow that converges to a critical point with at least an ``$L^1$'' rate must a posteriori be fully dominated by one of its ``fast'' unstable modes $\lambda_{I^*}$, with $I^* \leq I$. We obtain sharp decay rates that are sensitive to the dominating unstable eigenvalue (i.e., $e^{-\lambda_{I^*} t}$). This lets us invoke the contraction mapping from Section \ref{sec:existence} with much weaker a priori conditions. 

\begin{rema} \label{rema:characterization.equation}
	Recall, from \eqref{eq:critical.point.operator.with.errors}, that solutions $u$ of \eqref{eq:negative.gradient.flow} satisfy an evolution equation of the form
	\begin{equation} \label{eq:characterization.nonlinear}
		\tfrac{\partial}{\partial t} u = Lu + \langle \mathscr{N}(x, u, \nabla_g u), \nabla_g^2 u \rangle_g + \mathscr{Q}(x, u, \nabla_g u).
	\end{equation}
	For reasons that will become clearer later, we will seek to resort to \eqref{eq:characterization.nonlinear} instead of the variational equation  \eqref{eq:negative.gradient.flow} whenever possible.
\end{rema}

\begin{rema} \label{rema:characterization.regularity}
	When $V$ is a line bundle, all $C^{1,\theta}_P$ norms in this section can be replaced by spatial $C^1$ norms due to fully nonlinear parabolic PDE theory \cite{Wang92}.
\end{rema}

\begin{proof}[Proof of Theorem \ref{theo:characterization}]	
	It will be convenient to write $\lesssim$ for inequalities that hold up to multiplicative constants that may depend on $\cA$, $\Sigma$, $g$, $V$, $\theta$, $C_0$. Denote:
	\[ \sigma(t) := \Vert u(\cdot, t) \Vert_{C^2(\Sigma; V)}. \]
	Linear parabolic Schauder theory, \eqref{eq:characterization.assumption.c1theta},  \eqref{eq:characterization.assumption.small},  \eqref{eq:characterization.assumption.L1} imply 
	\begin{equation} \label{eq:characterization.initial.decay}
		\sigma(t) \lesssim \eps, \; \sigma \in L^1(\RR_-).
	\end{equation}
	For sufficiently small $\eps > 0$, the negative gradient flow nature of  \eqref{eq:negative.gradient.flow}, the criticality of $0$, the Legendre--Hadamard condition \eqref{eq:elliptic.functional.convexity} and the uniform $C^2$ control in \eqref{eq:characterization.initial.decay} imply, by G\aa{}rding's inequality, that
	\[ \cA(0) \geq \cA(u(\cdot, t)) \geq \cA(0) + C^{-1} \Vert \nabla_g u(\cdot, t) \Vert_{L^2(\Sigma; V)}^2 - C^{-1} \Vert u(\cdot, t) \Vert_{L^2(\Sigma; V)}^2 \]
	for a fixed $C > 0$, so
	\begin{equation} \label{eq:characterization.w12.l2}
		\Vert u(\cdot, t) \Vert_{W^{1,2}(\Sigma; V)} \lesssim \Vert u(\cdot, t) \Vert_{L^2(\Sigma; V)}, \; \forall t \leq 0.
	\end{equation}
	From now on don't need to use  \eqref{eq:negative.gradient.flow} again, and will just use \eqref{eq:characterization.nonlinear}. 
	
	Now \eqref{eq:characterization.nonlinear}, \eqref{eq:critical.point.equation.error.i}, \eqref{eq:critical.point.equation.error.ii} imply
	\begin{equation} \label{eq:characteration.error.bound}
		\Vert (\tfrac{\partial}{\partial t} - L) u(\cdot, t) \Vert_{L^2(\Sigma; V)} \lesssim \sigma(t) \Vert u(\cdot, t) \Vert_{W^{1,2}(\Sigma)} \lesssim \sigma(t) \Vert u(\cdot, t) \Vert_{L^2(\Sigma; V)}.
	\end{equation}
	Denote
	\begin{align*}
		U_-(t) & := \Vert \Pi_-(u(\cdot, t)) \Vert_{L^2(\Sigma; V)}, \\
		U_0(t) & := \Vert \Pi_0(u(\cdot, t)) \Vert_{L^2(\Sigma; V)}, \\
		U_+(t) & := \Vert (\operatorname{Id} - \Pi_- - \Pi_0)(u(\cdot, t)) \Vert_{L^2(\Sigma; V)},
	\end{align*}
	so that
	\[ \Vert u(\cdot, t) \Vert_{L^2(\Sigma; V)}^2 = U_-(t)^2 + U_0(t)^2 + U_+(t)^2. \]
	Proceeding as in \cite[Lemma 5.5]{AngenentDaskalopoulosSesum15}, we see that \eqref{eq:characteration.error.bound} implies:
	\begin{equation} \label{eq:characterization.ddt.pminus}
		\tfrac{d}{dt} U_- + \lambda_I U_- \gtrsim -\sigma \Vert u(\cdot, t) \Vert_{L^2(\Sigma; V)},
	\end{equation}
	\begin{equation} \label{eq:characterization.ddt.p0}
		|\tfrac{d}{dt} U_0| \lesssim \sigma \Vert u(\cdot, t) \Vert_{L^2(\Sigma; V)},
	\end{equation}
	\begin{equation} \label{eq:characterization.ddt.pplus}
		\tfrac{d}{dt} U_+ + \lambda_{I+K+1} U_+ \lesssim \sigma \Vert u(\cdot, t)  \Vert_{L^2(\Sigma; V)}.
	\end{equation}
	\begin{clai} \label{clai:characterization.pminus.dominant}
		$U_+ + U_0 \lesssim \sigma U_-$, $t \leq 0$.
	\end{clai}
	\begin{proof}[Proof of claim]
		We apply the ODE Lemma \ref{lemm:mz.ode} with $x = U_0$, $y = U_+$, $z = U_-$. This already implies $U_+ \lesssim \sigma (U_- + U_0)$. Our claim will follow once we can show that \eqref{eq:mz.ode.B} holds true. Note that this is trivially true in case $K = \dim \ker L = 0$.
		
		Suppose, for the sake of contradiction, that $K \geq 1$ and \eqref{eq:mz.ode.A} holds true instead instead of \eqref{eq:mz.ode.B}. Then $U_+ + U_- \lesssim \sigma U_0$ and, therefore,
		\begin{equation} \label{eq:characterization.pminus.dominant.not}
			\Vert u(\cdot, t) \Vert_{L^2(\Sigma; V)} \lesssim U_0(t)
		\end{equation}
		for $t \leq -\tau$ and some $\tau > 0$.
		
		Linear parabolic Schauder theory, \eqref{eq:characterization.nonlinear}, and   \eqref{eq:characterization.pminus.dominant.not} imply
		\[ \sigma(t) \lesssim \int_{t-1}^t U_0(s) \, ds \leq \max_{[t-1,t]} U_0, \; t \leq -\tau. \]
		Therefore, \eqref{eq:characterization.ddt.p0}, \eqref{eq:characterization.pminus.dominant.not} together imply that $V(t) := \max_{[t-1,t]} U_0$ satisfies $V'(t) \lesssim V(t)^2$. Integrating and recalling the definition of $U_0$ 
		\[ \max_{[t-1,t]} U_0 \gtrsim |t|^{-1}, \; t \leq -\tau. \]
		Together with \eqref{eq:characterization.ddt.p0}, \eqref{eq:characterization.pminus.dominant.not}, again, this implies $U_0(t) \gtrsim |t|^{-1}$ for \emph{every} $t \leq -\tau$, which contradicts the $L^1$ finiteness in  \eqref{eq:characterization.initial.decay}. The claim follows. 
	\end{proof}

	Claim \ref{clai:characterization.pminus.dominant}  improves \eqref{eq:characterization.ddt.pminus} to:
	\begin{equation} \label{eq:characterization.ddt.pminus.1}
	\tfrac{d}{dt} U_- + \lambda_I U_- \gtrsim -\sigma U_-.
	\end{equation}
	At this point we may assume, without loss of generality, that $U_-(t) > 0$ for $t \leq 0$; otherwise, Claim \ref{clai:characterization.pminus.dominant} forces the trivial stationary situation $U_- \equiv U_0 \equiv U_+ \equiv 0$. 
	
	\begin{clai} \label{clai:characterization.strong} 
		$\sigma(t) \lesssim e^{-\lambda_I t}$, $t \leq 0$.
	\end{clai}
	\begin{proof}[Proof of claim]
		We first prove a weaker statement. For $a \in (0, -\lambda_I)$:
		\begin{equation} \label{eq:characterization.weak}
			\sigma(t) \leq C_a e^{a t}, \; t \leq 0.
		\end{equation}
		It follows from Claim \ref{clai:characterization.pminus.dominant}, \eqref{eq:characterization.initial.decay},  \eqref{eq:characterization.ddt.pminus}, that $(\log U_-)' \geq -a$ for $t \leq -\tau_a$, with $\tau_a$ large depending on $a$. Integrating,
		\[ \log U_-(-\tau_a) - \log U_-(t) = \int_t^{-\tau_a} (\log U_-)' \geq -a (t + \tau_a), \; t \leq -\tau_a. \]
		Rearranging, $U_-(t) \leq U_-(-\tau_a) e^{a (t+\tau_a)}$. Claim \ref{clai:characterization.pminus.dominant}, \eqref{eq:characterization.assumption.c1theta}, \eqref{eq:characterization.assumption.small},  and linear parabolic Schauder theory on \eqref{eq:characterization.nonlinear} now  imply \eqref{eq:characterization.weak}. 
		
		We now prove the strong bound. If $a \in (0, -\lambda_I)$,  \eqref{eq:characterization.ddt.pminus.1}, \eqref{eq:characterization.weak} imply
		\[ \log U_-(0) - \log (e^{\lambda_I t} U_-(t)) = \int_t^0 (\log (e^{\lambda_I s} U_-(s)))' \, ds \geq - \tfrac1a C_a. \]
		The claim follows by Claim \ref{clai:characterization.pminus.dominant}, \eqref{eq:characterization.assumption.c1theta}, \eqref{eq:characterization.assumption.small},  and linear parabolic Schauder theory on \eqref{eq:characterization.nonlinear}.
	\end{proof}

	For $j = 1, \ldots, I$ and $\lambda < 0$ denote:
	\[ u_j(t) := \langle u(\cdot, t), \varphi_j \rangle_{L^2(\Sigma; V)} \varphi_j, \]
	\[ S_{\geq \lambda} := \{ j \in \{1, \ldots, I \} : \lambda_j \geq \lambda \}. \]
	\[ U_{\geq \lambda}(t) := \Big[ \sum_{j \in S_{\geq \lambda}} \Vert u_j(t) \Vert_{L^2(\Sigma; V)}^2 \Big]^{1/2}, \]
	and so on for all symbols $<$, $=$, etc.

	\begin{clai} \label{clai:characterization.peel.irrelevant.terms}
		Suppose $I' \in \{ 1, \ldots, I \}$ is such that
		\begin{equation} \label{eq:characterization.peel.irrelevant.terms}
			j \in S_{\geq \lambda_{I'}} \implies \lim_{t \to -\infty} e^{\lambda_j t} u_j(t) = 0.
		\end{equation}
		Then $S_{< \lambda_{I'}} \neq \emptyset$ and $U_{\geq \lambda_{I'}} \lesssim \sigma U_{< \lambda_{I'}}$, $t \leq 0$.
	\end{clai}
	\begin{proof}[Proof of claim]
		We prove this by backward induction on $I'$. Let us see how the base case, $I' = I$, goes. Thus, assume \eqref{eq:characterization.peel.irrelevant.terms} holds for $I' = I$. For any $\lambda < \lambda_I$ and $\lambda \geq \max \{ \lambda_j : j \in S_{<\lambda_I} \}$, we have:
		\begin{equation} \label{eq:characterization.ddt.pI}
			|\tfrac{d}{dt} U_{=\lambda_I} + \lambda_I U_{=\lambda_I}| \lesssim \sigma U_-,
		\end{equation}
		\begin{equation} \label{eq:characterization.ddt.pIminus1}
			\tfrac{d}{dt} U_{<\lambda_I} + \lambda U_{<\lambda_I} \gtrsim - \sigma U_-;
		\end{equation}
		the second ODE being interpreted as vacuously true in case $S_{< \lambda_{I}} = \emptyset$.
		
		Note that $U_{=\lambda_I} \lesssim U_{<\lambda_I}$. If not, then \eqref{eq:characterization.ddt.pI} would imply 
		\[ | \tfrac{d}{dt} (e^{\lambda_I t} U_{=\lambda_I}) | \lesssim \sigma e^{\lambda_I t} U_- \lesssim \sigma e^{\lambda_I t} U_{=\lambda_I} \implies | \tfrac{d}{dt} \log (e^{\lambda_I t} U_{=\lambda_I})| \lesssim \sigma. \]
		Integrating, and using $\sigma \in L^1$ from  \eqref{eq:characterization.initial.decay}, we get $\lim_{t \to -\infty} e^{\lambda_I t} U_{=\lambda_I}(t) > 0$; this contradicts  \eqref{eq:characterization.peel.irrelevant.terms}. Therefore, $U_{=\lambda_I} \lesssim U_{<\lambda_I}$ as claimed. 
		
		As a consequence, $S_{< \lambda_I} \neq \emptyset$. Moreover, the ODE Lemma \ref{lemm:mz.ode} applied to \eqref{eq:characterization.ddt.pI}, \eqref{eq:characterization.ddt.pIminus1} with $x := e^{\lambda_I t} U_{=\lambda_I}$, $z := e^{\lambda_I t} U_{<\lambda_I}$, improves $U_{=\lambda_I} \lesssim U_{<\lambda_I}$ to $U_{=\lambda_I} \lesssim \sigma U_{<\lambda_I}$. This completes the base case of the backward induction. 
		
		For the general case, repeat with $I'$ instead of $I$ in \eqref{eq:characterization.ddt.pI}, \eqref{eq:characterization.ddt.pIminus1}.
	\end{proof}

	Let $I^* \in \{ 1, \ldots, I \}$ be the largest index for which \eqref{eq:characterization.peel.irrelevant.terms}  \emph{fails}. Then:
	\begin{equation} \label{eq:characterization.ddt.dominant}
		\tfrac{d}{dt} U_{\leq \lambda_{I^*}} + \lambda_{I^*} U_{\leq \lambda_{I^*}} \gtrsim -\sigma U_{\leq \lambda_{I^*}}.
	\end{equation}
	We have thus established that the analog of \eqref{eq:characterization.ddt.pminus.1} holds for the topmost modes that do not vanish at infinity, and Claim \ref{clai:characterization.peel.irrelevant.terms} now plays the role of Claim  \ref{clai:characterization.pminus.dominant}. Thus, arguing as in Claim  \ref{clai:characterization.strong}, we obtain the following sharp estimate:
	\begin{equation} \label{eq:characterization.dominant.sigma}
		\sigma(t) \lesssim e^{-\lambda_{I^*} t}.
	\end{equation}

	In all that follows, $\tau \geq 0$ and $t \leq 0$. Set
	\[ u^{(\tau)}(\cdot, t) := u(\cdot, t-\tau). \]
	Our assumption on $I^*$ guarantees that
	\begin{equation} \label{eq:characterization.dominant.RHS}
		\limsup_{\tau \to \infty} e^{-\lambda_{I^*} \tau} \Vert \Pi_-(u^{(\tau)}(\cdot, 0))\Vert_{L^2(\Sigma; V)} > 0.	
	\end{equation}
	
	\begin{clai} \label{clai:characterization.dominant.LHS}
		For every $t \leq 0$,
		\begin{equation} \label{eq:characterization.dominant.LHS}
			e^{-2\lambda_{I^*} \tau} \Vert u^{(\tau)}(\cdot, t) - \iota_-(u_1(-\tau), \ldots, u_I(-\tau))(\cdot, t) \Vert_{L^2(\Sigma; V)} \lesssim e^{-\lambda_I t}.	
		\end{equation}
	\end{clai}
	\begin{proof}[Proof of claim]
		First:
		\begin{equation} \label{eq:characterization.dominant.LHS.i}
			e^{-2\lambda_{I^*} \tau} (U_0 + U_+)(t-\tau) \lesssim e^{-2\lambda_{I^*} \tau} \sigma(t-\tau) U_-(t-\tau) \lesssim e^{-2\lambda_{I^*} t}.
		\end{equation}
		where we've used Claim \ref{clai:characterization.pminus.dominant} and \eqref{eq:characterization.dominant.sigma}. Second, for $j \in \{1, \ldots, I\}$, it is easy to see that Claims  \ref{clai:characterization.pminus.dominant}, \ref{clai:characterization.peel.irrelevant.terms} imply $\Vert \tfrac{d}{dt} u_j + \lambda_j u_j \Vert \lesssim \sigma U_{\leq \lambda_{I^*}}$. Multiplying through with $e^{\lambda_j t}$, integrating, and using \eqref{eq:characterization.dominant.sigma} again:
		\begin{equation} \label{eq:characterization.dominant.LHS.ii}
			e^{-2\lambda_{I^*} \tau} \Vert u_j(t-\tau) - e^{-\lambda_j t} u_j(-\tau) \Vert \lesssim e^{-\max \{\lambda_j, 2 \lambda_{I^*}\} t}
		\end{equation}
		when $\lambda_j \neq 2 \lambda_{I*}$, or, in case $\lambda_j = 2\lambda_{I^*}$,
		\begin{equation} \label{eq:characterization.dominant.LHS.ii.degenerate}
			e^{-2\lambda_{I^*} \tau} \Vert u_j(t-\tau) - e^{-\lambda_j t} u_j(-\tau) \Vert \lesssim |t| e^{-2\lambda_{I^*} t}.
		\end{equation}
		Combining \eqref{eq:characterization.dominant.LHS.i}, \eqref{eq:characterization.dominant.LHS.ii}, \eqref{eq:characterization.dominant.LHS.ii.degenerate} gives the result.
	\end{proof}

	Let $\delta_0 \in (0, -\lambda_I)$. Linear parabolic Schauder theory on \eqref{eq:characterization.nonlinear} promotes \eqref{eq:characterization.dominant.LHS} to
	\begin{equation} \label{eq:characterization.dominant.LHS.weighted}
		e^{-2\lambda_{I^*} \tau} \Vert u^{(\tau)} - \iota_-(u_1(-\tau), \ldots, u_I(-\tau)) \Vert_{C^{2,\theta,\delta_0}_P(\Sigma \times \RR_-; V)} \lesssim 1.
	\end{equation}
	The result follows from the uniqueness aspect of Theorem \ref{theo:existence} applied to $u^{(\tau)}$ after choosing a sufficiently large $\mu$, depending on the implicit constants of  \eqref{eq:characterization.dominant.RHS}, \eqref{eq:characterization.dominant.LHS.weighted}, and a sufficiently large $\tau$ from \eqref{eq:characterization.dominant.RHS}.
\end{proof}

We now seek to provide sufficient conditions that will guarantee the $L^1$ decay needed to apply Theorem \ref{theo:characterization}. To that end, we recall the notion of integrable critical points:

\begin{defi} \label{defi:integrable}
	The critical point $f=0$ of $\cA$ is said to be {\bf integrable} if for every $\phi \in \ker L$ there exists a family of $\{f_t\}_{0 < t < 1} \subset \cM^{1,\theta}(1)$ with $\cA(f_t) = \cA(0)$ and $\lim_{t \to 0} \Vert \tfrac{1}{t} f_t - \phi \Vert_{C^{1,\theta}(\Sigma; V)} = 0$.
\end{defi}

\begin{lemm}[cf. {\cite[Lemma 6.4, Part II]{Simon84}}]  \label{lemm:integrable.technical.lemma}
	Let $0$ be an integrable critical point of $\cA$. Fix $\theta \in (0, 1)$. There exist $\eps, c, \tau > 0$ such that if $\psi_* \in \cM^{1,\theta}(\eps)$ and $u : \Sigma \times \RR_- \to V$ is a smooth solution of \eqref{eq:negative.gradient.flow} with  $\Vert u \Vert_{C^{1,\theta}_P(\Sigma \times \RR_-; V)} < \eps$, then
	\begin{align} \label{eq:integrable.technical.lemma}
		& \Vert u - \psi_{**} \Vert_{C^{1,\theta}_P(\Sigma \times [-3\tau, -\tau]; V)} \\
		& \qquad \leq \tfrac12 \max \Big\{ c \Big[ \lim_{t \to -\infty} \cA(u(\cdot, t)) - \cA(\psi_*) \Big]_+^{1/2}, \Vert u - \psi_* \Vert_{C^{1,\theta}_P(\Sigma \times [-2\tau, 0]; V)} \Big\} \nonumber
	\end{align}
	for some $\psi_{**} \in \cM^{1,\theta}(1)$.
\end{lemm}
\begin{proof}
	Below, $\lesssim$ will be used for  inequalities that hold up to multiplicative constants that  depend on $\cA$, $\Sigma$, $g$, $V$, $\theta$, $C_0$. We will adapt Simon's proof from the elliptic and forward-parabolic settings in \cite{Simon84} to the backward parabolic setting. We argue by contradiction. If the conclusion were false, then there would exist sequences
	\begin{itemize}
		\item $\psi_*^{(k)} \in  \cM^{1,\theta}(1/k)$, and 
		\item $u^{(k)} : \Sigma \times \RR_- \to \RR$ of solutions to \eqref{eq:negative.gradient.flow} with 
		\[ \Vert u^{(k)} \Vert_{C^{1,\theta}_P(\Sigma \times \RR_-; V)} < 1/k, \]
	\end{itemize}
	so that
	\begin{align} \label{eq:integrable.technical.lemma.contradiction}
		& \inf \Big\{ \Vert u^{(k)} - \psi_{**}^{(k)} \Vert_{C^{1,\theta}_P(\Sigma \times [-3\tau, -\tau]; V)} : \psi_{**} \in \cM^{1,\theta}(1) \Big\} \\
		& > \tfrac12 \max \Big \{ k^{1/2} \Big[ \lim_{t \to -\infty} \cA(u^{(k)}(\cdot, t)) - \cA(\psi_*^{(k)}) \Big]_+^{1/2},  \Vert u^{(k)} - \psi_*^{(k)} \Vert_{C^{1,\theta}_P(\Sigma \times [-2\tau, 0]; V)} \Big\}. \nonumber  
	\end{align}
	It follows from \eqref{eq:negative.gradient.flow} that
	\begin{equation} \label{eq:integrable.technical.lemma.L2.int.leq.area.diff}
		\int_{-\infty}^0 \int_\Sigma \Vert \tfrac{\partial}{\partial t} u^{(k)}(\cdot, t) \Vert^2 \, d\mu_g \, dt \lesssim \lim_{t \to -\infty} \cA(u^{(k)}(\cdot, t)) - \cA(u^{(k)}(\cdot, 0)).
	\end{equation}
	By a crude estimate on $\cA(u^{(k)}(\cdot, 0)) - \cA(\psi_*^{(k)})$ this implies
	\begin{multline} \label{eq:integrable.technical.lemma.L2.int.leq.C1}
		\int_{-\infty}^0 \int_\Sigma \Vert \tfrac{\partial}{\partial t} u^{(k)}(\cdot, t) \Vert^2 \, d\mu_g \, dt \lesssim \Vert u^{(k)}(\cdot, 0) - \psi^{(k)}_* \Vert_{C^1(\Sigma; V)}^2 \\
		+ \lim_{t \to -\infty} \cA(u^{(k)}(\cdot, t)) - \cA(\psi_*^{(k)}).
	\end{multline}
	Therefore,
	\begin{multline} \label{eq:integrable.technical.lemma.L2.leq.C1}
		\Vert u^{(k)}(\cdot, t) - \psi^{(k)}_* \Vert_{L^2(\Sigma; V)}^2 \lesssim |t| \Big[ \Vert u^{(k)}(\cdot, 0) - \psi^{(k)}_* \Vert_{C^1(\Sigma; V)}^2 \\
		+ \lim_{t \to -\infty} \cA(u^{(k)}(\cdot, t)) - \cA(\psi_*^{(k)}) \Big], \; t \leq -1.
	\end{multline}
	At this point we will no longer need to use the variational structure of \eqref{eq:negative.gradient.flow}, and will instead use \eqref{eq:characterization.nonlinear}. Linear parabolic Schauder theory implies
	\begin{multline} \label{eq:integrable.technical.lemma.C1alpha.leq.C1}
		\Vert u^{(k)} - \psi^{(k)}_* \Vert_{C^{1,\theta}_P(\Sigma \times [t, -1]; V)}^2  \lesssim |t| \Big[ \Vert u^{(k)}(\cdot, 0) - \psi^{(k)}_* \Vert_{C^1(\Sigma; V)}^2 \\
		+ \lim_{t \to -\infty} \cA(u^{(k)}(\cdot, t)) - \cA(\psi_*^{(k)}) \Big], \; t \leq -1.
	\end{multline}
	Together with   \eqref{eq:integrable.technical.lemma.contradiction} applied with $\psi^{(k)}_{**} = \psi^{(k)}_*$, \eqref{eq:integrable.technical.lemma.C1alpha.leq.C1} implies
	\begin{equation} \label{eq:integrable.technical.lemma.C1alpha.leq.C1alpha}
		\Vert u^{(k)} - \psi^{(k)}_* \Vert_{C^{1,\theta}_P(\Sigma \times [t, -1]; V)}^2  \lesssim |t| \Vert u^{(k)} - \psi^{(k)}_* \Vert_{C^{1,\theta}_P(\Sigma \times [-2\tau, 0]; V)}^2, \; t \leq -1.
	\end{equation}
	Define
	\[ \beta^{(k)} := \Vert u^{(k)} - \psi^{(k)}_* \Vert_{C^{1,\theta}_P(\Sigma \times [-2\tau, 0]; V)}, \]
	\[ \hat{u}^{(k)} := (\beta^{(k)})^{-1} (u^{(k)} - \psi^{(k)}_*). \]
	Using linear parabolic Schauder theory, the $\hat{u}^{(k)}$ have uniform $C^{2,\theta}_P$ estimates as $k \to \infty$.  By Arzel\`a--Ascoli on \eqref{eq:integrable.technical.lemma.C1alpha.leq.C1alpha}, Fatou's lemma on \eqref{eq:integrable.technical.lemma.L2.int.leq.C1}, and $\beta^{(k)} \to 0$, we see that, after passing to a subsequence, $\hat{u}^{(k)}$ converges locally in $C^{1,\theta}_P$ to a function $\hat{u} : \Sigma \times \RR_- \to V$ which satisfies	
	\begin{equation} \label{eq:integrable.technical.lemma.blowup.limit.C1alpha}
		\Vert \hat{u} \Vert_{C^{1,\theta}_P(\Sigma \times [-2\tau, 0]; V)} = 1,
	\end{equation}
	\begin{equation} \label{eq:integrable.technical.lemma.blowup.limit.pde}
		\tfrac{\partial}{\partial t} \hat{u} = L \hat{u},
	\end{equation}
	\begin{equation} \label{eq:integrable.technical.lemma.blowup.limit.L2}
		\int_{-\infty}^0 \int_\Sigma \Vert \tfrac{\partial}{\partial t} \hat{u} \Vert^2 \, d\mu_g \, dt \lesssim 1,
	\end{equation}
	where $L$ is as in \eqref{eq:jacobi.operator}. It follows from \eqref{eq:integrable.technical.lemma.blowup.limit.pde}, \eqref{eq:integrable.technical.lemma.blowup.limit.L2} that
	\begin{equation} \label{eq:integrable.technical.lemma.blowup.limit.decomposition}
		\hat{u} = \hat{\phi} + \iota_-(\hat{\bm{a}})
	\end{equation}
	for $\hat{\phi} \in \ker L$ and $\bm{a} \in \RR^I$. By the $C^{1,\theta}_P(\Sigma \times [-3\tau, 0]; V)$ convergence $\hat{u}^{(k)} \to \hat{u}$, Claim \ref{clai:tau.zero} below (where $\xi \in (0,1)$ is yet to be determined), and \eqref{eq:integrable.technical.lemma.blowup.limit.C1alpha}:
	\begin{align*}
		& \Vert u^{(k)} - \psi^{(k)}_* - \beta^{(k)} \hat{\phi} \Vert_{C^{1,\theta}_P(\Sigma \times [-3\tau, -\tau]; V)} \\
			& \qquad = \beta^{(k)} \Vert (\beta^{(k)})^{-1}(u^{(k)} - \psi^{(k)}_*) - \hat{\phi} \Vert_{C^{1,\theta}_P(\Sigma \times [-3\tau, -\tau]; V)} \\
			& \qquad = \beta^{(k)} \Vert \hat{u} - \hat{\phi} \Vert_{C^{1,\theta}_P(\Sigma \times [-3\tau, -\tau]; V)} + o(\beta^{(k)}) \\
			& \qquad \leq \xi \beta^{(k)} \Vert \hat{u} - \hat{\phi} \Vert_{C^{1,\theta}_P(\Sigma \times [-2\tau, 0]; V)} + o(\beta^{(k)}) \\
			& \qquad = \xi \beta^{(k)} \Vert (\beta^{(k)})^{-1} (u^{(k)} - \psi^{(k)}_*) - \hat{\phi} \Vert_{C^{1,\theta}_P(\Sigma \times [-2\tau, 0]; V)} + o(\beta^{(k)}) \\
			& \qquad \leq \xi \beta^{(k)} + \xi \beta^{(k)} \Vert \hat{\phi} \Vert_{C^{1,\theta}(\Sigma; V)} + o(\beta^{(k)}).
	\end{align*}
	By elliptic theory, \eqref{eq:integrable.technical.lemma.blowup.limit.C1alpha}, and  \eqref{eq:integrable.technical.lemma.blowup.limit.decomposition} we can choose $\xi$ uniformly so that $\xi + \xi \Vert \hat{\psi} \Vert_{C^{1,\theta}(\Sigma)} \leq \tfrac14$. Thus:
	\[ \Vert u^{(k)} - \psi^{(k)}_* - \beta^{(k)} \hat{\phi} \Vert_{C^{1,\theta}_P(\Sigma \times [-3\tau, -\tau]; V)} \leq (\tfrac14 + o(1)) \beta^{(k)}. \]
	Together with Definition \ref{defi:integrable}, this contradicts \eqref{eq:integrable.technical.lemma.contradiction}.
\end{proof}

\begin{clai} \label{clai:tau.zero}
	Fix $\xi$, $\theta \in (0, 1)$. There exists $\tau > 0$ such that
	\begin{equation} \label{eq:tau.zero.claim}
		\Vert \iota_-(\bm{a}) \Vert_{C^{1,\theta}_P(\Sigma \times [-3\tau, -\tau]; V)} \leq \xi \cdot \Vert \iota_-(\bm{a}) \Vert_{C^{1,\theta}_P(\Sigma \times [-2\tau, 0]; V)},
	\end{equation}
	independently of $\bm{a} \in \RR^I$.
\end{clai}
\begin{proof}
	This is a straightforward consequence of the exponential decay as $t \to -\infty$ in \eqref{eq:iota.minus}.
\end{proof}

\begin{rema} \label{rema:integrable.technical.lemma.area}
	If $\cA(\psi_*) = \cA(0)$ in Lemma \ref{lemm:integrable.technical.lemma}, then we can guarantee that $\cA(\psi_{**}) = \cA(0)$. This follows because all perturbed solutions are produced by Definition \ref{defi:integrable}.
\end{rema}

\begin{prop} \label{prop:integrable}
	Let $0$ be an integrable critical point of $\cA$. Fix $\theta \in (0, 1)$. There exist $\eps$, $c$, $\kappa > 0$ such that if $u : \Sigma \times \RR_- \to \RR$ is a smooth solution of \eqref{eq:negative.gradient.flow} with $\Vert u \Vert_{C^{1,\theta}_P(\Sigma \times \RR_-; V)} < \eps$ and
	\begin{equation} \label{eq:integrable.assumption}
		\lim_{t \to -\infty} \cA(u(\cdot, t)) \leq  \cA(0),
	\end{equation}
	then there exists $\psi_* \in \cM^{1,\theta}(\eps)$ such that
	\begin{equation} \label{eq:integrable.conclusion}
		\Vert u - \psi_* \Vert_{C^{1,\theta,\kappa}_P(\Sigma; V)} \leq c.
	\end{equation}
\end{prop}
\begin{proof}
	We will iterate Lemma \ref{lemm:integrable.technical.lemma}. On every iteration we can estimate:
	\begin{align} 
		& \Vert \psi_* - \psi_{**} \Vert_{C^{1,\theta}(\Sigma; V)} \nonumber \\
		& \qquad \leq \Vert u(\cdot, -\tau) - \psi_* \Vert_{C^{1,\theta}(\Sigma; V)} + \Vert u(\cdot, -\tau) - \psi_{**}  \Vert_{C^{1,\theta}(\Sigma; V)} \nonumber \\
		& \qquad \leq \tfrac32 \Vert u - \psi_* \Vert_{C^{1,\theta}_P(\Sigma \times [-2\tau, 0]; V)}. \label{eq:integrable.iteration.i}
	\end{align}
	Suppose $\mu \in (0, 1)$ is yet to be determined, and set $\eps := \mu \eps_0$. From Lemma \ref{lemm:integrable.technical.lemma} and   \eqref{eq:integrable.iteration.i}, applied with $\psi_{*,0} := 0$, there exists a critical point $\psi_{*,1}$ such that
	\[ \Vert u - \psi_{*,1} \Vert_{C^{1,\theta}_P(\Sigma \times [-3\tau, -\tau]; V)} \leq \tfrac12  \mu \eps_0, \]
	\[ \Vert \psi_{*,1} - \psi_{*,0} \Vert_{C^{1,\theta}(\Sigma; V)} \leq \tfrac32 \mu \eps_0. \]
	By Remark \ref{rema:integrable.technical.lemma.area}, $\cA(\psi_{*,1}) = \cA(0)$. Iterating indefinitely, we obtain critical points $\psi_{*,k}$ with $\cA(\psi_{*,k}) = \cA(0)$, and 
	\[ \Vert u - \psi_{*,k} \Vert_{C^{1,\theta}_P(\Sigma \times [-(k+2)\tau, -k\tau]; V)} \leq  2^{-k} \mu \eps_0, \]
	\[ \Vert \psi_{*,k} - \psi_{*,k-1} \Vert_{C^{1,\theta}(\Sigma; V)} \leq 3 \cdot 2^{-k} \mu \eps_0. \]
	Using this geometric decay, we find that there exists $\psi_{*,\infty} \in \cM^{1,0}(3\eps)$ as asserted. The result follows with $\eps/3$ in place of $\eps$.
\end{proof}

\begin{rema} \label{rema:integrable.sufficient}
	We list two general sufficient conditions for assumption \eqref{eq:integrable.assumption} in Proposition \ref{prop:integrable} to hold:
	\begin{enumerate}
		\item The critical point $0$ is {\bf nondegenerate}, i.e., $\dim \ker L = 0$. It is then simple to see that there exists $\varepsilon > 0$ such that $\cM^{1,\theta}(\varepsilon) = \{ 0 \}$ (e.g., this follows immediately from the analysis in Section \ref{sec:notation.critical.points}). As a side consequence, the limiting $\psi_*$ is  $\psi_* = 0$.
		\item The integrand $A(x, z, q)$ in \eqref{eq:elliptic.functional} is an {\bf analytic} function of $z$, $q$. It is then easy to see that $\cA_{\operatorname{fin}}$ in Section \ref{sec:notation.critical.points} is constant on a neighborhood of the origin so, by \eqref{eq:critical.point.space}, the left and right hand sides of \eqref{eq:integrable.assumption} are equal. (Assumption \eqref{eq:integrable.assumption} also follows from the much stronger \L{}ojasiewicz--Simon inequality \cite{Simon83}. However, the conclusion of Proposition \ref{prop:integrable} certainly needn't hold if we're not near an integrable critical point; see Appendix \ref{sec:slow.convergence} for examples of arbitrarily slow convergence.)
	\end{enumerate}
\end{rema}

%% file: mean-curvature-flow.tex
%!TEX root = main.tex

\section{Mean curvature flow} \label{sec:mcf}

\begin{lemm} \label{lemm:mcf.measure.closeness}
	Let $S \subset (M, \overline{g})$ be a  closed and smoothly embedded minimal submanifold. For $\theta \in (0, 1)$, $0 < \sigma < \tfrac12 \tau$, $\eps > 0$, there exists $\delta > 0$ so that if $(\Sigma_t)_{0 \leq t \leq \tau}$ is a mean curvature flow that stays $\delta$-close to $S$ in the sense of measures, i.e., for all $t \in [0, \tau]$, 
	\begin{equation} \label{eq:mcf.measure.closeness} \left| \int_{\Sigma_t} f \, d\mu_{\overline{g} \llcorner \Sigma_t} - \int_S f \, d\mu_{\overline{g} \llcorner S} \right| \leq \delta \Vert f \Vert_{C^0(M)}, \; \forall f \in C^0(M),	
	\end{equation}
	then, for all $t \in [\sigma, 2\sigma]$, $\Sigma_t$ is a graph of some function on $S$ with values in the normal bundle $NS$ and  $C^{2,\theta}(S; NS)$ norm $< \eps$.
\end{lemm}	
\begin{proof}
	Let $m$ be the dimension of $S$. Denote Gaussian density ratios for points $x \in M$, a surface $T$, and a scale $r > 0$, by
	\[ \Theta(x; T, r) := (4\pi r^2)^{-m/2} \int_T \exp(-d(x, y)^2/4 r^2) \, d\mu_T(y). \]
	Fix $\eta > 0$. Because $S$ is closed and smoothly embedded, there exists $r_0 \in (0,1)$ such that $\Theta(\cdot; S, r_0) \leq 1 + \eta$ on $M$. Thus:
	\[ \Theta(\cdot ; \Sigma_t, r_0) \leq 1 + \eta + (4 \pi r_0^2)^{-m/2} \delta \leq 1 + 2 \eta, \; t \leq 0, \]
	provided $\delta > 0$ is sufficiently small. White's local regularity theorem \cite{White05} for point with Gaussian density close to one yields uniform estimates on the second fundamental forms of $\Sigma_t$, $t \in [\sigma, 2\sigma]$. The fact that the hypersurfaces $\Sigma_t$ are all graphical over $S$ with small $C^2$ norm follows by a straightforward contradiction argument given that we now know uniform curvature bounds and uniform measure closeness to $S$; the $C^2$ norm is improved to a $C^{2,\theta}$ norm by standard regularity theory.
\end{proof}

Lemma \ref{lemm:mcf.measure.closeness} implies the following results, whose proofs will be given momentarily:

\begin{theo} \label{theo:mcf.characterization}
	Let $S \subset (M, \overline{g})$ be a closed and smoothly embedded minimal submanifold. There exists $\delta > 0$ such that if $(\Sigma_t)_{t \leq 0}$ is a mean curvature flow that stays measure theoretically $\delta$-close to $S$ in the sense of \eqref{eq:mcf.measure.closeness}, and
		\begin{equation} \label{eq:mcf.characterization.assumption}
			\int_{-\infty}^0 \dist_{\overline{g}}(\Sigma, \Sigma_t) \, dt < \infty,
		\end{equation}
		then there exists $\tau \geq 0$ and $\bm{a} \in \RR^I$ such that $(\Sigma_{t-\tau})_{t \leq 0}$ coincides with $\mathscr{S}(\bm{a})$ from Theorem \ref{theo:existence}. 
\end{theo}

By virtue of the a posteriori $C^{2,\theta}$ bound in Lemma \ref{lemm:mcf.measure.closeness}, the distance function in \eqref{eq:mcf.characterization.assumption} can be one of several equivalent distance type functions, but for simplicity we take the supremum distance.

\begin{prop} \label{prop:mcf.integrable}
	Let $S \subset (M, \overline{g})$ be an integrable, closed, smoothly embedded minimal surface. There exists $c, \delta, \kappa > 0$ such that if $(\Sigma_t)_{t \leq 0}$ is an ancient mean curvature flow which is measure theoretically $\delta$-close to $S$ in the sense of \eqref{eq:mcf.measure.closeness}, and
	\begin{equation} \label{eq:mcf.integrable.assumption}
		\lim_{t \to -\infty} \operatorname{Area}_{\overline{g}}(\Sigma_t) \leq \operatorname{Area}_{\overline{g}}(\Sigma),
	\end{equation}
	then there exists a possibly different closed, smoothly embedded minimal surface $S_*$ such that $\Sigma_t$, $t \leq -1$, is a graph of some function on $S_*$ with values in the normal bundle $NS_*$ and $C^{2,\theta}(S_*; NS_*)$ norm $< c e^{\kappa t}$.
\end{prop}

We remind the reader that Remark \ref{rema:integrable.sufficient} describes situations where the area condition of Proposition  \ref{prop:mcf.integrable} is met.

We now turn to the proof of Theorem  \ref{theo:mcf.characterization} and Proposition \ref{prop:mcf.integrable}.  We would like to apply Theorem \ref{theo:characterization} and Proposition \ref{prop:integrable}. Unfortunately, mean curvature flow is \emph{not} the gradient flow, in the sense of \eqref{eq:negative.gradient.flow}, for the elliptic \emph{nonparametric} area functional; see Remark \ref{rema:mcf.vs.abstract}. It is, however, an evolution of the form \eqref{eq:characterization.nonlinear} for the $L$, $\mathscr{N}$, $\mathscr{Q}$ that come from the nonparametric area functional. Therefore, Theorem \ref{theo:characterization} and Proposition \ref{prop:integrable} will apply provided we ensure the validity of all steps where the variational implications \eqref{eq:negative.gradient.flow} were used, and not just the general evolution \eqref{eq:characterization.nonlinear}.

\begin{proof}[Proof of Theorem \ref{theo:mcf.characterization}]
	By Lemma \ref{lemm:mcf.measure.closeness}, we can write the $\Sigma_t$ as small $C^{2,\theta}$ graphs of $u(\cdot, t)$, where $u : S \times \RR_- \to NS$ is a normal bundle valued solution of the nonparametric mean curvature flow equation. So, we seek to apply the proof of Theorem \ref{theo:characterization}. Inspecting the proof, we see that the negative gradient flow equation \eqref{eq:negative.gradient.flow} was only used to derive \eqref{eq:characterization.w12.l2}, which nevertheless continues to hold for our parametric gradient flow, as we are $C^1$-near a minimal submanifold. The remainder of the proof applies verbatim.
\end{proof}

\begin{proof}[Proof of Proposition \ref{prop:mcf.integrable}]
		By Lemma \ref{lemm:mcf.measure.closeness}, we can write the $\Sigma_t$ as small $C^{2,\theta}$ graphs of $u(\cdot, t)$, where $u : S \times \RR_- \to NS$ is a normal bundle valued solution of the nonparametric mean curvature flow equation. So, we seek to apply the proof of Proposition \ref{prop:integrable}. Inspecting the proof, we see that the negative gradient flow equation \eqref{eq:negative.gradient.flow} was only used to derive \eqref{eq:integrable.technical.lemma.L2.int.leq.area.diff}, which nevertheless continues to hold, since $\Vert \tfrac{\partial}{\partial t} u \Vert^2 \, d\mu_g$ is bounded by a \emph{fixed constant} times $\Vert \mathbf{H}_{\Sigma_t} \Vert^2 \, d\mu_{\overline{g} \llcorner \Sigma_t}$, as we are $C^1$-near a fixed submanifold. Thus, by the first variation formula,
		\begin{align*}
			\int_{-\infty}^0 \int_S \Vert \tfrac{\partial}{\partial t} u^{(k)}(\cdot, t) \Vert^2 \, d\mu_g \, dt 
			& \lesssim \int_{-\infty}^0 \int_S \Vert \mathbf{H}_{\Sigma_t} \Vert^2 \, d\mu_{\overline{g}  \llcorner \Sigma_t} \, dt \\
			& \lesssim \int_{-\infty}^0 \tfrac{d}{dt} \operatorname{Area}_{\overline{g}}(\Sigma_t) \, dt \\
			& = \lim_{t \to -\infty} \operatorname{Area}_{\overline{g}}(\Sigma_t) - \operatorname{Area}_{\overline{g}}(\Sigma_0).
		\end{align*}
		The remainder of the proof applies verbatim.
\end{proof}

Before proving our Allard-type characterization of ancient mean curvature flows in the sphere, we prove the following toy result:

\begin{prop} \label{prop:mcf.hemisphere}
	Let $(\Sigma_t)_{t \leq 0}$ be an ancient mean curvature flow of hypersurfaces embedded in a round hemisphere $\overline{\SS^n_+}$. If
	\begin{equation} \label{eq:mcf.hemisphere.assumption}
		\lim_{t \to -\infty} \operatorname{Area}(\Sigma_t) < 2 \operatorname{Area}(\SS^{n-1}),
	\end{equation}
	then $(\Sigma_t)_{t \leq 0}$ is the  steady $\partial \SS^n_+$ or spheres of latitude flowing out of it.
\end{prop}

It would be interesting to know whether \eqref{eq:mcf.hemisphere.assumption} can be relaxed.

\begin{proof}[Proof of Proposition \ref{prop:mcf.hemisphere}]
	We seek to employ Theorem   \ref{theo:mcf.characterization} and Proposition \ref{prop:mcf.integrable} with $S = \partial \SS^n_+$. Indeed, $I = \ind(L) = 1$ on $\partial \SS^n_+$ by \cite[Proposition 5.1.1]{Simons68}, and it is trivial (e.g., by direct construction) to see that this one-parameter family of ancient flows corresponds to one of spheres of latitude.
	
	\begin{clai} \label{clai:mcf.hemisphere.limit}
		$\lim_{t \to -\infty} \Sigma_t = \partial \SS^n_+$ in the sense of measures.
	\end{clai}
	\begin{proof}[Proof of claim]
		Consider any sequence $t_i \to -\infty$ and the sequence of translated flows $\Sigma^{(i)}_t := \Sigma_{t+t_i}$. By Brakke's compactness theorem and the uniform boundedness of areas, $(\Sigma^{(i)}_t)_t$ converges subsequentially to an integral eternal Brakke flow with constant area equal to $\lim_{t \to -\infty} \operatorname{Area}(\Sigma_t)$. Since the area is constant, the integral Brakke flow is supported on a stationary integral varifold $V$, with $\support \Vert V \Vert \subset \overline{\SS^n_+}$. It is easy to see that all such varifolds will, in fact, satisfy $\support \Vert V \Vert \subset \partial \SS^n_+$ (use, e.g., the conformal Killing field normal to $\partial \SS^n_+$). By the constancy theorem for integral varifolds \cite{SimonGMT} and \eqref{eq:mcf.hemisphere.assumption} it follows that $V$ is $\partial \SS^n_+$ with multiplicity one.
	\end{proof}
	
	Claim \ref{clai:mcf.hemisphere.limit} and Proposition \ref{prop:mcf.integrable}, together, show that $\Sigma_t \to S$ exponentially as $t \to -\infty$. We are using the well-known fact that equatorial $\SS^m \subset \SS^n$ are integrable (see \cite[Proposition 5.1.1]{Simons68} for the dimension of the space of Jacobi fields, which trivially matches the space of $\SS^m$'s generated by symmetries). Therefore, Theorem  \ref{theo:mcf.characterization} yields the result, since the Morse index of an equatorial $\SS^m \subset \SS^n$ is $n-m$ \cite[Proposition 5.1.1]{Simons68}.
\end{proof}

\begin{lemm} \label{lemm:mcf.analytic.uniqueness}
	Let $(\Sigma_t)_{t \leq 0}$ be an ancient mean curvature flow of closed submanifolds in a real analytic manifold $(M, \overline{g})$. Suppose that there exists a closed minimal submanifold $S_0$ and times $t_i \to -\infty$ so that $\lim_i \Sigma_{t_i} = S_0$ in $C^{2,\theta}$. For all small $\eta > 0$, there exists $\tau \geq 0$ so that $\Sigma_t$ is $\eta$-close to $S_0$, in $C^{2,\theta}(S_0; NS_0)$,  for all $t \leq -\tau$.
\end{lemm}
\begin{proof}
	We seek to invoke, backward in time, the \L{}ojasiewicz--Simon inequality \cite[Theorem 3]{Simon83} on the real analytic manifold $(M, \overline{g})$. Recall that its content is that there exist $\eta_0 > 0$, $\mu \in (0, 1)$ (depending on $S_0$, $M$, $\overline{g}$) such that if $f : S_0 \to NS_0$ has $\Vert f \Vert_{C^{2,\theta}(S_0; NS_0)} < \eta_0$, then
	\begin{equation} \label{eq:mcf.allard.lojasiewicz.simon}
	\Vert \mathbf{H}(f) \Vert_{L^2(S_0; NS_0)} \geq |\operatorname{Area}(S_0) - \operatorname{Area}(f)|^{1-\mu}.
	\end{equation}
	Here we're writing $\mathbf{H}(f)$ and $\operatorname{Area}(f)$ for the mean curvature vector and area of the graph of $f$. 
	
	Without loss of generality, $\eta < \eta_0$. For each $i$, let $t_i'$ be the first time $> t_i$ at which $\Sigma_t$ cannot be written as a graph $u(\cdot, t)$ over $S_0$ with $\Vert u(\cdot, t) \Vert_{C^{2,\theta}(S_0; NS_0)} < \tfrac12 \eta$, or $t_i' = 0$ if no such time exists. Clearly, $t_i'$ is nonincreasing in $i$. Our lemma is equivalent to showing
	\begin{equation} \label{eq:mcf.allard.time.intervals}
		\lim_i t_i' > -\infty.
	\end{equation} 
	Assume  \eqref{eq:mcf.allard.time.intervals} is false. In that case, we first show that:
	\begin{equation} \label{eq:mcf.allard.L2.bound}
		\liminf_i \int_{t_i}^{t_i'} \Vert \tfrac{\partial}{\partial t} u(\cdot, t) \Vert_{L^2(S_0; NS_0)} > 0.
	\end{equation}
	If \eqref{eq:mcf.allard.L2.bound} were false, then $\lim_i \Sigma_{t_i} =  S_0$ in $C^{2,\theta}$ (and thus $L^2$) would imply 
	\[ \lim_i \sup_{[t_i, t_i']} \Vert u(\cdot, t) \Vert_{L^2(S_0; NS_0)} \to 0. \]
	Note that $t_i' - t_i \to \infty$ due to $\lim_i \Sigma_{t_i} = S$, which is a minimal submanifold. Therefore, $t_i' > t_i + 1$ for sufficiently large $i$. By Schauder theory, we can estimate
	\[ \Vert u(\cdot, t_i') \Vert_{C^{2,\theta}(S_0; NS_0)}^2 \lesssim \int_{t_i'-1}^{t_i'} \Vert u(\cdot, t) \Vert_{L^2(S_0; NS_0)}^2 \, dt \to 0, \]
	which contradicts $\Vert u(\cdot, t_i') \Vert_{C^{2,\theta}(S_0; NS_0)} = \tfrac12 \eta$. Therefore, \eqref{eq:mcf.allard.L2.bound} is true. By the parametric mean curvature flow evolution equation and \eqref{eq:mcf.allard.lojasiewicz.simon}, there exists a constant $c$ (close to $1$) such that
	\begin{align*}
	& \tfrac{d}{dt} (\operatorname{Area}(S_0) - \operatorname{Area}(\Sigma_t))^\mu \\
	& \qquad \geq c \mu (\operatorname{Area}(S_0) - \operatorname{Area}(\Sigma_t))^{\mu-1} \Vert \mathbf{H}(u(\cdot, t)) \Vert_{L^2(S_0; NS_0)} \Vert \tfrac{\partial}{\partial t} u(\cdot, t) \Vert_{L^2(S_0, NS_0)} \\
	& \qquad \geq c \mu \Vert \tfrac{\partial}{\partial t} u(\cdot, t) \Vert_{L^2(S_0; NS_0)}.
	\end{align*}
	Integrating,
	\begin{align*}
	& (\operatorname{Area}(S_0) - \operatorname{Area}(\Sigma_{t_i'}))^\mu - (\operatorname{Area}(S_0) - \operatorname{Area}(\Sigma_{t_i}))^\mu \\
	& \qquad \geq c\mu \int_{t_i}^{t_i'} \Vert \tfrac{\partial}{\partial t} u(\cdot, t) \Vert_{L^2(S; NS)} \, dt.
	\end{align*}
	By \eqref{eq:mcf.allard.L2.bound}, this right hand side has a positive $\liminf$, which contradicts that the left hand side gives terms of a convergent series, due to the monotonicity of area. Thus, \eqref{eq:mcf.allard.time.intervals} holds true. 
\end{proof}

\begin{proof}[Proof of Theorem \ref{theo:mcf.allard}]
	First we show:
	
	\begin{clai} \label{clai:mcf.allard.limit}
		$\lim_{t \to -\infty} \Sigma_t$ exists in  $C^{2,\theta}$ and is an equatorial $\SS^m$.
	\end{clai}
	\begin{proof}
		First, pick an arbitrary $t_i \to -\infty$. By Brakke's compactness theorem \cite{Brakke78}, $(\Sigma_{t+t_i})_{t \leq 0}$ has a subsequence which converges in a measure theoretic sense to an integral Brakke flow whose area is a constant $\leq (1 + \delta) \operatorname{Area}(\SS^m)$, due to by \eqref{eq:mcf.allard.assumption}. Therefore, this is a static flow of a stationary integral varifold $V$ in $\SS^n$. By Allard's regularity theorem \cite{SimonGMT} applied to the stationary cone $0 \# V$, if $\delta$ is sufficiently small then $0 \# V$ must be a smooth cone, so $V$ must be an equatorial $\SS^m$ with multiplicity one. Therefore, all backward in time subsequential measure theoretic limits are multiplicity one equatorial $\SS^m$'s. Note that then White's regularity theorem by way of Lemma \ref{lemm:mcf.measure.closeness} promote the convergence to $C^{2,\theta}$.  By Lemma \ref{lemm:mcf.analytic.uniqueness}, $\Sigma_t$ is close to a \emph{fixed} $\SS^m$ for sufficiently negative $t$. But then Proposition \ref{prop:mcf.integrable} promotes this to full convergence in $C^{2,\theta}$, as $t \to -\infty$, to a fixed $\SS^m$.
	\end{proof}
	
	The result now follows as it did in Proposition \ref{prop:mcf.hemisphere} with the combination of Proposition \ref{prop:mcf.integrable} and Theorem \ref{theo:mcf.characterization}.  The $n-m$ potential flow directions of $\SS^m \subset \SS^n$ are predicted by the Morse index, which is $n-m$ \cite[Proposition 5.1.1]{Simons68}.
\end{proof}

\begin{proof}[Proof of Corollary  \ref{coro:csf}]
	It suffices to show that \eqref{eq:csf.length.assumption} implies \eqref{eq:mcf.allard.assumption}. Recall that
	\[ \lim_{t \to -\infty} \operatorname{Length}(\Gamma_t) - \operatorname{Length}(\Gamma_0) = \int_{-\infty}^0 \int_{\Gamma_t} \kappa_{\Gamma_t}^2 \, d\ell_{\Gamma_t}. \]
	Therefore, by virtue of \eqref{eq:csf.length.assumption}, there exists a sequence $t_i \to -\infty$ with
	\[ \lim_i \int_{\Gamma_{t_i}} \kappa_{\Gamma_{t_i}}^2 \, d\ell_{\Gamma_{t_i}} = 0. \]
	By the Sobolev embedding theorem and Allard's varifold compactness theorem \cite{SimonGMT}, after passing to a subsequence $\{ t_{i'} \}_{i'} \subset \{ t_i \}_i$, the curves $\Gamma_{t_{i'}}$ converge in $C^{1,\theta}$, $\theta \in (0, \tfrac12)$, to a stationary $C^{1,1/2}$ curve $\Gamma_*$. The only stationary $C^{1,1/2}$ curves inside $\SS^2$ are equators and multiples thereof, but a simple degree argument shows that embedded curves cannot converge in $C^0$ to an equator with multiplicity greater than one. Therefore,
	\[ \lim_{i'} \operatorname{Length}(\Gamma_{t_{i'}}) = \operatorname{Length}(\SS^1). \]
	This implies \eqref{eq:mcf.allard.assumption}, and the result follows by Theorem \ref{theo:mcf.allard}.
\end{proof}

Below we prove a technical lemma needed for Corollary \ref{coro:mcf.s3}:

\begin{lemm} \label{lemm:willmore.weak}
	Let $T$ be a 2-dimensional stationary integral varifold in $\SS^3$. If $\operatorname{Area}(T) \leq 2\pi^2$ and its associated $\ZZ_2$ chain $[T]$ has $\partial [T] = 0$, then $T$ is a multiplicity one equator or Clifford torus.
\end{lemm}
\begin{proof}
	If $T$ is smooth, the result follows by Marques--Neves's  \cite{MarquesNeves14} resolution of the Willmore conjecture: smooth minimal surfaces in $\SS^3$ have area $4\pi$ (equator), $2 \pi^2$ (Clifford torus), or larger. 
	
	We will show that $T$ is, indeed, smooth by arguing by contradiction. If $T^2 := T$ were singular, then the 3-dimensional stationary cone $C^3 := 0 \# T^2$ would be one where the origin is not an isolated singularity. If $x_0 \neq 0$ denotes a singular point of $C^3$, then, by the monotonicity formula, the densities of $C^3$ at $x_0$ and the origin $0$ satisfy (see \cite{SimonGMT}):
	\begin{equation} \label{eq:mcf.s3.cone.i}
	\Theta^3(C^3, x_0) \leq \Theta^3(C^3, 0) = \tfrac{1}{4\pi} \operatorname{Area}(T^2) \leq \tfrac12 \pi.
	\end{equation}
	Let $\overline{C}^3$ be a tangent cone to $C^3$ at $x_0$. Automatically, $\overline{C}^3 \cong \overline{C}^2 \times \RR$ for some stationary 2-dimensional cone $\overline{C}^2 \subset \RR^3$. Let $T^1 \subset \SS^2$ be the link of $\overline{C}^2$. It has
	\begin{equation} \label{eq:mcf.s3.cone.ii} 	
	\operatorname{Length}(T^1) = 2\pi \Theta^2(\overline{C}^2, 0) = 2\pi \Theta^3(C^3, x_0) \leq 2 \pi \cdot \tfrac12 \pi = \pi^2.
	\end{equation}
	If $T^1$ is smooth, then $\pi^2 < 4\pi$ and \eqref{eq:mcf.s3.cone.ii} imply $T^1 \cong \SS^1$ with multiplicity one, so  $\overline{C}^2 \cong \RR^2$ with multiplicity one, so $\overline{C}^3 \cong \RR^3$ with multiplicity one; this violates the singular nature of $x_0 \in C^3$ by Allard's theorem \cite{SimonGMT}.
	
	Therefore, $T^1$ has to be singular, too. We repeat our previous argument. Let $x_1 \neq 0$ denote a singular point of the 2-dimensional cone $C^2 := 0 \# T^1 \subset \RR^3$. By the same argument as in \eqref{eq:mcf.s3.cone.i}, and using \eqref{eq:mcf.s3.cone.ii},
	\[ \Theta^2(C^2, x_1) \leq \Theta^2(C^2, 0) = \tfrac{1}{2\pi} \operatorname{Length}(T^1) \leq  \tfrac12 \pi. \]
	Let $\widehat{C}^2$ be a tangent cone to $C^2$ at $x_1$. Automatically, $\widehat{C}^2 \cong \widehat{C}^1 \times \RR$ for some 1-dimensional stationary cone $\widehat{C}^1 \subset \RR^2$ with
	\begin{equation} \label{eq:mcf.s3.cone.iii}
	\Theta^1(\widehat{C}^1, 0) = \Theta^2(C^2, x_1) \leq \tfrac12 \pi.
	\end{equation}
	It is well known that all  1-dimensional stationary cones are unions of $k \geq 2$ half-rays and have $\Theta^1(\widehat{C}^1, 0) = \tfrac12 k$. We have $k \leq 3$ by \eqref{eq:mcf.s3.cone.iii}. Moreover, $k$ is even because $\widehat{C}^1$ is obtained by various blow ups of the $\ZZ_2$ cycle $T$. Therefore, $k=2$. This means $\widehat{C}^1 \cong \RR$ with multiplicity one and $\widehat{C}^2 \cong \RR^2$ with multiplicity one; this violates the singular nature of $x_1 \in C^2$ by Allard's theorem.
\end{proof}

\begin{proof}[Proof of Corollary \ref{coro:mcf.s3}]
	We proceed as in the proof of Theorem \ref{theo:mcf.allard}. Note that both $\SS^2$ and the Clifford torus are integrable minimal surfaces; for the latter see \cite[Theorem 10]{WuYiLawson71} which shows the space of Jacobi fields is 4-dimensional on Clifford tori, which matches the dimension of the space  of Clifford tori ($\approx \RR \PP^2 \times \RR \PP^2$). Therefore, in either case we will have a \emph{unique} backward in time limit (by repeating the argument from Claim \ref{clai:mcf.allard.limit}, which involved Lemma \ref{lemm:mcf.analytic.uniqueness} and Proposition \ref{prop:mcf.integrable}), as long as we can show:
	
	\begin{clai} \label{clai:mcf.s3.limit}
		If  $\delta$ in \eqref{eq:mcf.s3.assumption} is small enough, then every backward subsequential limit of $(\Sigma_t)_{t \leq 0}$ is a multiplicity one equator or Clifford torus.
	\end{clai}
	\begin{proof}[Proof of Claim \ref{clai:mcf.s3.limit}]
		We first prove a weaker result; namely, that the claim holds true if we replace \eqref{eq:mcf.s3.assumption} by $\lim_{t \to -\infty} \operatorname{Area}(\Sigma_t) \leq 2\pi^2$. Arguing as in Claim \ref{clai:mcf.allard.limit}, we know that any subsequential limit of the translated flows is an eternal integral Brakke flow with constant area $\leq 2\pi^2$. By \cite[Theorem 4.2]{White09}, the limiting stationary integral varifold $T$ has an associated $\ZZ_2$ chain $[T]$ with $\partial [T] = 0$. Therefore, the claim follows from Lemma \ref{lemm:willmore.weak} above.
		
		We now prove the general claim. Suppose $(\Sigma^{(k)}_t)_{t \leq 0}$ is a sequence of ancient flows in $\SS^3$ satisfying \eqref{eq:mcf.s3.assumption} with $\delta = \delta_k \to 0$. Let $T^{(k)}$ be some backward subsequential limit of $(\Sigma^{(k)}_t)_{t \leq 0}$. 
		
		By White's mean curvature flow theorem \cite[Theorem 4.2]{White09}, $T^{(k)}$ is a stationary integral varifold in $\SS^3$ with $\operatorname{Area}(T^{(k)}) \leq 2\pi^2 + \delta_k$ and a  corresponding $\ZZ_2$ chain $[T^{(k)}]$ with $\partial [T^{(k)}] = 0$. Passing to a subsequence and using Allard's theorem \cite{SimonGMT} and White's enhanced convergence theorem \cite[Theorem 1.1]{White09}, $\lim_{k} T^{(k)} = T$, a stationary integral varifold in $\SS^3$ with $\operatorname{Area}(T) \leq 2\pi^2$ and a corresponding $\ZZ_2$ $\ZZ_2$ chain $[T]$ with $\partial [T] = 0$. By Lemma \ref{lemm:willmore.weak}, $T$ is a multiplicity one equator or Clifford torus. Therefore, by Allard's theorem \cite{SimonGMT}, each $T^{(k)}$, with $k$ sufficiently large, is smooth. Thus, $T^{(k)}$ is also a multiplicity one equator or Clifford torus; this follows from the integrability of equators and Clifford tori and the discussion regarding  \eqref{eq:critical.point.space}. Thus,
		\[ \lim_{t \to -\infty} \operatorname{Area}(\Sigma^{(k)}_t) = \operatorname{Area}(T^{(k)}) \leq 2\pi^2. \]
		The claim follows from the weaker result we initially proved.
	\end{proof}

	The result again follows as it did in Theorem \ref{theo:mcf.allard}.
\end{proof}

%% file: slow-examples.tex
%!TEX root = main.tex

\section{Examples with arbitrarily slow convergence} \label{sec:slow.convergence}

We describe examples of ancient mean curvature flows which converge to their backward-in-time limits arbitrarily slowly. In particular they converge slower than polynomially, in contrast with the integrable case and the real analytic case.  We wish to point out that Carlotto--Chodosh--Rubinstein \cite{CarlottoChodoshRubinstein15} recently used the \L{}ojasiewicz--Simon inequality for an interesting systematic study of the speed of convergence of Yamabe flows, another example of a parabolic flow. Their argument can be reasonably expected to be adaptable to our setting, too. However, for our purposes there are fairly explicit examples of slow ancient mean curvature flows, which we construct below. 

We will construct examples $(\SS^2, g)$, where the metric $g$ is rotationally symmetric and, away from two antipodal points $p$, $p'$,
\[ g := ds^2 + e^{2f(s)} \, d\theta^2, \]
for $(s, \theta) \times (-2, 2) \times \SS^1 \approx \SS^2 \setminus \{ p, p'\}$ and $f : (-2, 2) \to \RR$ is to be determined. Let $\tau : (0, 1] \to (-\infty, 0]$ be smooth and such that $\lim_{s \to 0} \tau(s) = -\infty$, $\tau(1) = 0$, $\tau' > 0$, and
\[ 
\int_0^1 \frac{d\sigma}{\tau'(\sigma)} < \infty.
\]
Such a function $\tau(s)$ can be prescribed as a time of arrival function for a curve in a rotationally symmetric space that is $s$ units away from the backward-in-time limit geodesic. We point out that our admissible time of arrival functions include, for instance, the sub-polynomial function $\tau(s) := \log |s|$. It is a straightforward exercise to check that
\[ f(s) := \int_0^s \frac{d\sigma}{\tau'(\sigma)} \]
is such that $t \mapsto \{ s : \tau(s) = t \}$ is an curve shortening flow with time of arrival function $\tau$. Of course, $f$ can be extended to $(-2, 2)$ in such a way so that the two-sphere closes up smoothly.

%% file: ode-lemma.tex
%!TEX root = main.tex

\section{An ODE lemma}  \label{sec:mz.ode}

We point out that the ODE lemma of Merle--Zaag \cite[Lemma A.1]{MerleZaag98} holds true without certain assumptions they made (namely, that $x$, $z \to 0$ as $s \to -\infty$, or that $y(s_j) \to 0$ along \emph{all} sequences $s_j \to -\infty$):

\begin{lemm} \label{lemm:mz.ode}
	Suppose $x$, $y$, $z : (-\infty, 0] \to [0, \infty)$ be absolutely continuous functions such that
	\begin{equation} \label{eq:mz.ode.nonzero}
		x + y + z > 0,
	\end{equation}
	\begin{equation} \label{eq:mz.ode.asymptotics}
		\liminf_{s \to -\infty} y(s) = 0,
	\end{equation}
	and, for some $\eps > 0$,
	\begin{equation} \label{eq:mz.ode.system}
		\begin{array}{ccc}
			|x'| \leq \eps (x + y + z), \\
			y' + y \leq \eps(x + z), \\
			z' - z \geq - \eps(x + y).
		\end{array}
	\end{equation}
	There exist $\eps_0 > 0$, $c > 0$, such that if $\eps \leq \eps_0$, then 
	\begin{equation} \label{eq:mz.ode.stable}
		y \leq 2 \eps (x + z) \text{ on } (-\infty, 0],
	\end{equation}
	and one of the following holds:
	\begin{equation} \label{eq:mz.ode.A}
		\text{either } \exists s_* \in (-\infty, 0] \text{ such that } z \leq 8 \eps x \text{ on } (-\infty, s_*],
	\end{equation}
	\begin{equation} \label{eq:mz.ode.B}
		\text{or } x \leq c \eps z \text{ on } (-\infty, 0].
	\end{equation}
\end{lemm}
\begin{proof}
	Conclusion \eqref{eq:mz.ode.stable} follows as in \cite[p. 172]{MerleZaag98}. Indeed, we claim that $\beta := y - 2 \eps (x + z) \leq 0$. If this were false, there would exist $s^* \leq 0$ with $\beta(s^*) > 0$. Then computing as in \cite[p. 172]{MerleZaag98}, $\beta > 0 \implies \beta' \leq 0$. In particular, $\beta \geq \beta(s^*) > 0$ on $(-\infty, s^*]$, contradicting $\liminf_{s \to -\infty} \beta(s) \leq 0$, which follows from \eqref{eq:mz.ode.asymptotics}. 
	
	Now either there exists $s_* \in (-\infty, 0]$ such that $z(s_*) < 8 \eps x(s_*)$, or $8 \eps x \leq z$ on $(-\infty, 0]$. In the first case, we proceed as in \cite[p. 173]{MerleZaag98} to show that $z \leq 8 \eps x$ everywhere on $(-\infty, s_*]$. In the second case, we can proceed as follows. First, note that $z' \geq \tfrac14 z$. Then,
	\begin{equation} \label{eq:mz.ode.B.i}
		z(s) \geq \tfrac14 \int_{-\infty}^s z, \; s \leq 0.
	\end{equation}
	Thus, there exists a sequence $s_i \to -\infty$ with $z(s_i) \to 0$ as $i \to \infty$. Thus, $x(s_i) \leq (8 \eps)^{-1} z(s_i) \to 0$ along the same sequence. Note that $x' \leq (2\eps + \tfrac14) z$. By the fundamental theorem of calculus, the dominated convergence theorem, and \eqref{eq:mz.ode.B.i},
	\[ x(s) = \lim_{i \to \infty} \left[ x(s_i) + (2\eps + \tfrac14) \int_{s_i}^s z \right] = (2\eps + \tfrac14) \int_{-\infty}^s z \leq (1 + 8 \eps) z(s). \]
	Bootstrapping this improved bound on $x$ in terms of $z$ into the estimate for $x'$, and proceeding with the same exact argument, the result follows with $c = 8\eps(2 + 8\eps)$. 
\end{proof}

%% file: regularity.tex
%!TEX root = main.tex

\section{Parabolic Schauder theory} \label{sec:regularity}

We collect here some facts regarding regularity theory for parabolic systems that we need to use. We will work on $\RR^{n+1}_+ = \RR^n \times \RR_+$ with $\RR^Q$-valued systems of the form $u_t = \cL u$ where:
\begin{equation} \label{eq:regularity.operator}
	\cL u := \sum_{i,j=1}^n a_{ij} D_{ij} u + \sum_{i=1}^n b_i D_i u + c u.
\end{equation}
The coefficients $a_{ij}$, $b_i$, $c : \RR^{n+1}_+ \to E := \operatorname{End}(\RR^Q)$ are such that:
\begin{equation} \label{eq:regularity.operator.elliptic}
	\sum_{i,j=1}^n \langle a_{ij} v, v \rangle \tau_i \tau_j \geq \lambda |\tau|^2 |v|^2, \; \forall (x, t) \in \RR^{n+1}_+, \; \tau \in \RR^n, \; v \in \RR^Q,
\end{equation}
\begin{equation} \label{eq:regularity.operator.holder}
	\Vert a_{ij} \Vert_{C^{0,\theta}_P(\RR^{n+1}_+; E)}, \; \Vert b_i \Vert_{C^{0,\theta}_P(\RR^{n+1}_+; E)}, \; \Vert c \Vert_{C^{0,\theta}_P(\RR^{n+1}_+; E)} \leq \Lambda,
\end{equation}
for some fixed constants $\lambda > 0$, $\Lambda > 0$, $\theta \in (0, 1)$. We recall that, in this setting, the G\aa{}rding inequality for elliptic systems remains valid; see the discussion near \cite[(1.8)']{Simon83}. (Going from this Euclidean setting to the curved setting is standard; we are working on smooth Riemannian manifolds whose connections factor into the coefficients $a_{ij}$, $b_i$, $c$ in ways that are allowed by \eqref{eq:regularity.operator.elliptic}, \eqref{eq:regularity.operator.holder}.) 

We now state the ``interior'' Schauder estimate for parabolic equations. References include Schlag \cite{Schlag96} and Simon \cite{Simon97}:

\begin{theo}[$C^{2,\theta}$-$C^\theta$-$L^\infty$ interior Schauder estimate] \label{theo:regularity.schauder}
	If $u : \RR^{n+1}_+ \to \RR^Q$ is smooth, then
	\begin{multline} \label{eq:regularity.schauder}
		[\tfrac{\partial}{\partial t} u]_{C^{\theta}_P(\RR^n \times [1, \infty); \RR^Q)} + [D^2 u]_{C^{\theta}_P(\RR^n \times [1,\infty); \RR^Q)} \\
		\leq C \big( [(\tfrac{\partial}{\partial t} - \cL) u]_{C^{\theta}_P(\RR^{n+1}_+; \RR^Q)} + \Vert u \Vert_{L^\infty(\RR^{n+1}_+; \RR^Q)} \big)
	\end{multline}
	for some constant $C = C(n, \lambda, \Lambda, \theta) > 0$.
\end{theo}

Note that one has the global interpolation inequality
\begin{equation} \label{eq:regularity.interpolation}
	\Vert w \Vert_{L^\infty(\RR^n \times \RR)} \leq \eps [w]_{C^\theta_P(\RR^n \times \RR)} + C(n, \eps, q) \Vert w \Vert_{L^q(\RR^n \times \RR)},
\end{equation}
for all $q \in [1, \infty)$, $\eps > 0$. A classical absorption and localization  argument that combines \eqref{eq:regularity.interpolation} with standard H\"older interpolation inequalities (\cite[Lemma 7]{Schlag96}; see also \cite[(1.5)]{Simon97}) yields:

\begin{theo}[$C^{2,\theta}$-$C^\theta$-$L^q$ interior Schauder estimate] \label{theo:regularity.schauder.lq}
	If $u : \RR^{n+1}_+ \to \RR^Q$ is smooth, $q \in [1, \infty)$, and $B_1$ is a unit ball in $\RR^n$, then
	\begin{multline} \label{eq:regularity.schauder.lq}
				[\tfrac{\partial}{\partial t} u]_{C^{\theta}_P(B_1 \times [1,2]; \RR^Q)} + [D^2 u]_{C^{\theta}_P(B_1 \times [1,2]; \RR^Q)} \\
		\leq C \big( [(\tfrac{\partial}{\partial t} - \cL) u]_{C^{\theta}_P(B_1 \times [0, 2]; \RR^Q)} + \Vert u \Vert_{L^q(B_1 \times [0,2]; \RR^Q)} \big)
	\end{multline}
	for some constant $C = C(n, \lambda, \Lambda, \theta, q) > 0$.
\end{theo}

%% file: main.bbl
\begin{thebibliography}{10}

\bibitem{AllardAlmgren81}
W.~K. Allard and F.~J. Almgren, Jr.
\newblock On the radial behavior of minimal surfaces and the uniqueness of
  their tangent cones.
\newblock {\em Ann. of Math. (2)}, 113(2):215--265, 1981.

\bibitem{AngenentDaskalopoulosSesum15}
S.~B. Angenent, P.~Daskalopoulos, and N.~Sesum.
\newblock Unique asymptotics of ancient convex mean curvature flow solutions.
\newblock {\em J. Differential Geom.}, 111(3):381--455, 2019.

\bibitem{AngenentDaskalopoulosSesum18}
S.~B. {Angenent}, P.~{Daskalopoulos}, and N.~{Sesum}.
\newblock {Uniqueness of two-convex closed ancient solutions to the mean
  curvature flow}.
\newblock {\em Ann. of Math. (2)}.
\newblock To appear.

\bibitem{BourniLangfordTinaglia}
T.~{Bourni}, M.~{Langford}, and G.~{Tinaglia}.
\newblock {A collapsing ancient solution of mean curvature flow in
  {$\mathbb{R}^3$}}.
\newblock {\em arXiv e-prints}, page arXiv:1705.06981, May 2017.

\bibitem{Brakke78}
K.~A. Brakke.
\newblock {\em The motion of a surface by its mean curvature}, volume~20 of
  {\em Mathematical Notes}.
\newblock Princeton University Press, Princeton, N.J., 1978.

\bibitem{BrendleChoi19}
S.~Brendle and K.~Choi.
\newblock Uniqueness of convex ancient solutions to mean curvature flow in
  {$\Bbb R^3$}.
\newblock {\em Invent. Math.}, 217(1):35--76, 2019.

\bibitem{BrendleChoi18}
S.~{Brendle} and K.~{Choi}.
\newblock {Uniqueness of convex ancient solutions to mean curvature flow in
  higher dimensions}.
\newblock {\em Geom. Topol.}
\newblock To appear.

\bibitem{BryanIvakiScheuer}
P.~{Bryan}, M.~N. {Ivaki}, and J.~{Scheuer}.
\newblock {On the classification of ancient solutions to curvature flows on the
  sphere}.
\newblock {\em arXiv e-prints}, page arXiv:1604.01694, April 2016.

\bibitem{BryanLouie16}
P.~Bryan and J.~Louie.
\newblock Classification of convex ancient solutions to curve shortening flow
  on the sphere.
\newblock {\em J. Geom. Anal.}, 26(2):858--872, 2016.

\bibitem{CaffarelliHardtSimon84}
L.~Caffarelli, R.~Hardt, and L.~Simon.
\newblock Minimal surfaces with isolated singularities.
\newblock {\em Manuscripta Math.}, 48(1-3):1--18, 1984.

\bibitem{CarlottoChodoshRubinstein15}
A.~Carlotto, O.~Chodosh, and Y.~A. Rubinstein.
\newblock Slowly converging {Y}amabe flows.
\newblock {\em Geom. Topol.}, 19(3):1523--1568, 2015.

\bibitem{ChodoshSchulze}
O.~Chodosh and F.~Schulze.
\newblock Uniqueness of asymptotically conical tangent flows.
\newblock {\em arXiv e-prints}, page arXiv:1901.06369, Jan 2019.

\bibitem{ChoiHaslhoferHershkovits}
K.~{Choi}, R.~{Haslhofer}, and O.~{Hershkovits}.
\newblock {Ancient low entropy flows, mean convex neighborhoods, and
  uniqueness}.
\newblock {\em arXiv e-prints}, page arXiv:1810.08467, Oct 2018.

\bibitem{ColdingMinicozzi15}
T.~H. Colding and W.~P. Minicozzi, II.
\newblock Uniqueness of blowups and {{\L}}ojasiewicz inequalities.
\newblock {\em Ann. of Math. (2)}, 182(1):221--285, 2015.

\bibitem{ColdingMinicozzi18a}
T.~H. Colding and W.~P. Minicozzi, II.
\newblock Dynamics of closed singularities.
\newblock {\em Ann. Inst. Fourier (Grenoble)}, 69(7):2973--3016, 2019.

\bibitem{ColdingMinicozzi18b}
T.~H. Colding and W.~P. Minicozzi, II.
\newblock Wandering singularities.
\newblock {\em arXiv preprint arXiv:1809.03585}, Sep 2018.

\bibitem{DaPratoLunardi88}
G.~Da~Prato and A.~Lunardi.
\newblock Stability, instability and center manifold theorem for fully
  nonlinear autonomous parabolic equations in {B}anach space.
\newblock {\em Arch. Rational Mech. Anal.}, 101(2):115--141, 1988.

\bibitem{DaskalopoulosdelPinoSesum}
P.~Daskalopoulos, M.~del Pino, and N.~Sesum.
\newblock Type {II} ancient compact solutions to the yamabe flow.
\newblock {\em J. Reine Angew. Math.}, 2018(738):1--71, 2018.

\bibitem{DaskalopoulosHamiltonSesum10}
P.~Daskalopoulos, R.~Hamilton, and N.~Sesum.
\newblock Classification of compact ancient solutions to the curve shortening
  flow.
\newblock {\em J. Differential Geom.}, 84(3):455--464, 2010.

\bibitem{GilbargTrudinger01}
D.~Gilbarg and N.~S. Trudinger.
\newblock {\em Elliptic partial differential equations of second order}.
\newblock Classics in Mathematics. Springer-Verlag, Berlin, 2001.
\newblock Reprint of the 1998 edition.

\bibitem{HaslhoferHershkovits16}
R.~Haslhofer and O.~Hershkovits.
\newblock Ancient solutions of the mean curvature flow.
\newblock {\em Comm. Anal. Geom.}, 24(3):593--604, 2016.

\bibitem{WuYiLawson71}
Wu-yi Hsiang and H.~Blaine Lawson, Jr.
\newblock Minimal submanifolds of low cohomogeneity.
\newblock {\em J. Differential Geom.}, 5:1--38, 1971.

\bibitem{Huisken90}
G.~Huisken.
\newblock Asymptotic-behavior for singularities of the mean-curvature flow.
\newblock {\em J. Differential Geom.}, 31(1):285--299, 1990.

\bibitem{HuiskenSinestrari15}
G.~Huisken and C.~Sinestrari.
\newblock Convex ancient solutions of the mean curvature flow.
\newblock {\em J. Differential Geom.}, 101(2):267--287, 2015.

\bibitem{Lunardi95}
A.~Lunardi.
\newblock {\em Analytic semigroups and optimal regularity in parabolic
  problems}.
\newblock Modern Birkh\"{a}user Classics. Birkh\"{a}user/Springer Basel AG,
  Basel, 1995.
\newblock [2013 reprint of the 1995 original].

\bibitem{MarquesNeves14}
F.~C. Marques and A.~Neves.
\newblock Min-max theory and the {W}illmore conjecture.
\newblock {\em Ann. of Math. (2)}, 179(2):683--782, 2014.

\bibitem{MerleZaag98}
F.~Merle and H.~Zaag.
\newblock Optimal estimates for blowup rate and behavior for nonlinear heat
  equations.
\newblock {\em Comm. Pure Appl. Math.}, 51(2):139--196, 1998.

\bibitem{Schlag96}
W.~Schlag.
\newblock Schauder and {$L^p$} estimates for parabolic systems via {C}ampanato
  spaces.
\newblock {\em Comm. Partial Differential Equations}, 21(7-8):1141--1175, 1996.

\bibitem{Schulze14}
F.~Schulze.
\newblock Uniqueness of compact tangent flows in mean curvature flow.
\newblock {\em J. Reine Angew. Math.}, 2014(690):163--172, 2014.

\bibitem{Simon83}
L.~Simon.
\newblock Asymptotics for a class of nonlinear evolution equations, with
  applications to geometric problems.
\newblock {\em Ann. of Math. (2)}, 118(3):525--571, 1983.

\bibitem{SimonGMT}
L.~Simon.
\newblock {\em Lectures on geometric measure theory}, volume~3 of {\em
  Proceedings of the Centre for Mathematical Analysis, Australian National
  University}.
\newblock Australian National University, Centre for Mathematical Analysis,
  Canberra, 1983.

\bibitem{Simon84}
L.~Simon.
\newblock Isolated singularities of extrema of geometric variational problems.
\newblock In {\em Harmonic mappings and minimal immersions ({M}ontecatini,
  1984)}, volume 1161 of {\em Lecture Notes in Math.}, pages 206--277.
  Springer, Berlin, 1985.

\bibitem{Simon97}
L.~Simon.
\newblock Schauder estimates by scaling.
\newblock {\em Calc. Var. Partial Differential Equations}, 5(5):391--407, 1997.

\bibitem{Simons68}
J.~Simons.
\newblock Minimal varieties in riemannian manifolds.
\newblock {\em Ann. of Math. (2)}, 88:62--105, 1968.

\bibitem{Urbano90}
F.~Urbano.
\newblock Minimal surfaces with low index in the three-dimensional sphere.
\newblock {\em Proc. Amer. Math. Soc.}, 108(4):989--992, 1990.

\bibitem{Wang92}
L.~Wang.
\newblock On the regularity theory of fully nonlinear parabolic equations. {I}.
\newblock {\em Comm. Pure Appl. Math.}, 45(1):27--76, 1992.

\bibitem{Wang11}
X.-J. Wang.
\newblock Convex solutions to the mean curvature flow.
\newblock {\em Ann. of Math. (2)}, 173(3):1185--1239, 2011.

\bibitem{White03}
B.~White.
\newblock The nature of singularities in mean curvature flow of mean-convex
  sets.
\newblock {\em J. Amer. Math. Soc.}, 16(1):123--138, 2003.

\bibitem{White05}
B.~White.
\newblock A local regularity theorem for mean curvature flow.
\newblock {\em Ann. of Math. (2)}, 161(3):1487--1519, 2005.

\bibitem{White09}
B.~White.
\newblock Currents and flat chains associated to varifolds, with an application
  to mean curvature flow.
\newblock {\em Duke Math. J.}, 148(1):41--62, 2009.

\end{thebibliography}
